\documentclass{article}
\usepackage{graphicx} 
\usepackage{amsmath,amsfonts,amsthm,amssymb}
\usepackage{mathrsfs}
\usepackage{mathtools}
\usepackage[dvipsnames]{xcolor}
\usepackage[margin=1in]{geometry}
\usepackage{bm}
\usepackage[colorlinks=true,
            linkcolor=blue,     
            citecolor=blue,    
            filecolor=blue,  
            urlcolor=blue       
]{hyperref}
\usepackage{cite}
\usepackage{authblk}

\numberwithin{equation}{section}

\newtheorem{theorem}{Theorem}[section]
\newtheorem{lemma}[theorem]{Lemma}
\newtheorem{proposition}[theorem]{Proposition}

\theoremstyle{definition}

\renewcommand{\d}{\ensuremath{\mathrm{d}}}
\newcommand{\R}{\ensuremath{\mathbb{R}}}
\newcommand{\N}{\ensuremath{\mathbb{N}}}
\newcommand{\Z}{\ensuremath{\mathbb{Z}}}

\newcommand{\EE}{\ensuremath{\mathcal{E}}}
\newcommand{\EEE}{\ensuremath{\mathscr{E}}}

\newcommand{\cB}{\ensuremath{\mathcal{B}}}
\newcommand{\cH}{\ensuremath{\mathcal{H}}}
\newcommand{\cL}{\ensuremath{\mathcal{L}}}
\newcommand{\cT}{\ensuremath{\mathcal{T}}}

\newcommand{\eps}{\ensuremath{\varepsilon}}
\newcommand{\tbar}{\ensuremath{\vert\mkern-3mu\vert\mkern-3mu\vert}}

\title{Well-posedness and stability of the self-similar profile for a thin-film equation with gravity\thanks{The authors thank \emph{Yakine Bahri} for discussions at an initial stage of this work. \emph{MVG} appreciates discussions with \emph{Daniel Matthes} during the \emph{Lorentz Center Workshop ``Nonlinear Diffusion Models: Analytical \& Numerical Challenges''}, specifically regarding reference \cite{parsch2025globalexponentialstabilitystationary}. \emph{MVG} is grateful to the \emph{University of Victoria} for its kind hospitality. \emph{SI} acknowledges the pleasant hospitality of \emph{Delft University of Technology}. This publication is part of the project \emph{Codimension two free boundary problems} (with project number \emph{VI.Vidi.223.019}) of the research program \emph{ENW Vidi}, which is financed by the \emph{Dutch Research Council} (\emph{NWO}).\\
\emph{2020 Mathematics Subject Classification}: 35C06, 35K65, 35K35, 76D08 \\
\emph{Keywords}: Self-similar asymptotics, thin-film equation, gravity, convergence rate}}
\author{Manuel V.~Gnann\footnote{Delft Institute of Applied Mathematics, Faculty of Electrical Engineering, Mathematics and Computer Science, Delft University of Technology, Mekelweg 4, 2628CD Delft, Netherlands, M.V.Gnann@tudelft.nl} \quad and \quad Slim Ibrahim\footnote{Department of Mathematics and Statistics, University of Victoria, BC Canada, ibrahims@uvic.ca}}
\date{\today}

\allowdisplaybreaks

\begin{document}

\maketitle

\begin{abstract}
We consider the thin-film equation with linear mobility and a stabilizing second-order porous-medium type term modeling gravity. The model admits self-similar solutions, and our goal is to analyze their stability. We reformulate the problem in mass–Lagrangian coordinates and exploit the underlying gradient-flow structure of the equation with respect to a weighted $L^2$
 inner product, where the weight is given by the self-similar source-type profile. This framework allows us to establish a coercivity result for the Hessian (the linearization around the self-similar solution) in a suitably weighted inner product. As a consequence, we prove the convergence of perturbations toward the self-similar profile at an algebraic rate of order $t^{-\frac15}$, in arbitrary scales of weighted Sobolev norms. The analysis relies on maximal-regularity estimates for the linearized evolution, combined with appropriate estimates for the nonlinear terms.

\medskip

Notably, beyond perturbative regimes and in contrast to previous results for the thin-film equation (convergence to the Smyth–Hill profile) or the porous-medium equation (convergence to the Barenblatt–Pattle solution), our analysis does not rely on an explicit (algebraic) representation of the self-similar profile. Instead, it is based solely on a systematic use of the ordinary differential equation satisfied by the self-similar solution, together with a careful analysis of its boundary asymptotics. As a result, we expect that the approach developed here can serve as a flexible toolbox for the study of more general classes of equations and for the stability analysis of special solutions in future work.
\end{abstract}
\tableofcontents
\section{Introduction and heuristic motivation}
\subsection{Thin-film equation with gravity and self-similar solutions}
We consider the thin-film equation with linear mobility and stabilizing gravity term according to
\begin{subequations}\label{tfe-free}
\begin{equation}\label{tfe-stable}
h_t + (h h_{yyy})_y - (h^3 h_y)_y = 0 \quad \text{in} \quad \{h > 0\}
\end{equation}
for which we assume a zero contact-angle condition
\begin{equation}\label{tfe-contact}
h_y = 0 \quad \text{on} \quad \partial\{h > 0\}.
\end{equation}
\end{subequations}
In what follows, we prove well-posedness of compactly-supported solutions to \eqref{tfe-free} and analyze their intermediate time asymptotics as $t \to \infty$. We therefore use the transformations
\begin{equation}\label{self-var-stable}
\xi \coloneqq e^{-\frac s 5} y, \qquad h(t,y) = e^{-\frac s 5} u(s,\xi), \qquad s \coloneqq \ln (t+1),
\end{equation}
which are motivated by the scaling of the equation $t \sim y^5$, and passage to the logarithmic time variable $s$. We then have
\[
h = e^{-\frac s 5} u, \qquad h_t = e^{-s} (e^{-\frac s 5} u_s) - \frac \xi 5 e^{-s} (e^{-\frac s 5} u)_\xi -\frac{e^{-\frac 6 5 s}}5u = e^{-\frac 6 5 s} u_s - \frac 1 5 e^{-\frac 6 5 s} (\xi u)_\xi, \qquad \partial_y = e^{-\frac s 5} \partial_\xi,
\]
so that equation~\eqref{tfe-stable} changes to
\begin{subequations}\label{tfe-self-free}
\begin{equation}\label{tfe-self}
u_s + (u u_{\xi\xi\xi})_\xi - \frac 1 4 (u^4)_{\xi\xi} -\frac 1 5 (\xi u)_\xi = 0 \quad \mbox{in} \quad \{u > 0\}.
\end{equation}
and the boundary condition \eqref{tfe-contact} now reads
\begin{equation}\label{tfe-self-contact}
u_\xi = 0 \quad \text{on} \quad \partial\{u > 0\}.
\end{equation}
\end{subequations}
The self-similar source-type solution is a stationary solution to \eqref{tfe-self}, i.e., it satisfies $u(s,\xi) = U(\xi)$, so that after one integration in $\xi$ on assuming sufficient decay as $\xi \to \pm \infty$ we have 
\begin{equation}\label{self-int-stable}
U''' - U^2 U' - \frac \xi 5 = 0 \quad \mbox{in} \quad \{U > 0\}.
\end{equation}
One further integration leads to
\begin{equation}\label{self-int-stable-2}
U'' + \frac 1 3 (U_0^3 - U^3) - \frac{\xi^2}{10} - U_0'' = 0 \quad \mbox{in} \quad \{U > 0\},
\end{equation}
where $U(0) = U_0$ and $U''(0) = U_0''$. Equation~\eqref{self-int-stable} admits a nonnegative classical solution with compact support $[-\ell,\ell]$ for some $\ell > 0$, where $U' = 0$ on $\partial\{U > 0\}$, see Beretta's work \cite[Theorem~1.1]{beretta1997selfsimilar}. This solution is even (cf.~\cite[\S3--4]{beretta1997selfsimilar}) and satisfies $U > 0$ in $(-\ell,\ell)$ (cf.~\cite[Lemma~2.2]{beretta1997selfsimilar}) and $U \in C^\infty([-\ell,\ell])$. Note that smoothness on $\{U > 0\} = (-\ell,\ell)$ is a consequence of a standard theory on ordinary differential equations (ODEs), and smoothness on $[-\ell,\ell]$ follows from \eqref{self-int-stable-2} and an iteration argument. We mention that a more refined boundary analysis of the self-similar solution also for nonlinear mobility exponents is provided in \cite{MajdoubTayachi2022}, while the unstable case has been analyzed in \cite{SlepcevPugh2005}.

\medskip

The crucial mathematical ingredient for proving well-posedness and stability of perturbations of $U$ in a suitable set of coordinates is the coercivity of the Hessian when viewing the dynamics \eqref{tfe-self-free} as a gradient flow. This will first be motivated at a heuristic level (without introducing all necessary functional analysis) in \S\ref{sec-gf-standard}--\ref{sec-stability-formal}. The functional-analytic setting we actually use, including a formal derivation of suitable linear estimates, is then presented in \S\ref{sec-formal-funcana}. Our rigorous results are stated in \S\ref{sec-main} and the proofs are provided in the subsequent sections \S\ref{sec-linear}--\ref{sec-stability}.

\subsection{Heuristics: the standard gradient-flow formulation\label{sec-gf-standard}}
We introduce the free energy functional
\begin{equation}\label{energy}
\EEE[u] \coloneqq \int_{\{u > 0\}} \Big(\frac 1 2 u_\xi^2 + \frac{1}{12} u^4 - \frac 1 3 U_0^3 u + \frac{\xi^2}{10} u + U_0'' u\Big) \, \d \xi
\end{equation}
and the inner product (the \emph{metric})
\begin{equation}\label{inner}
\langle v, w \rangle_u \coloneqq \int_{\{u > 0\}} u^{-1} (\partial_\xi^{-1} v) (\partial_\xi^{-1} w) \, \d\xi, \quad \text{where } \int v \, \d \xi = \int w \, \d \xi = 0.
\end{equation}
We refer to the work of Otto \cite{Otto2001} for details regarding the underlying geometry, and to the work of Benamou and Brenier \cite{BenamouBrenier2000} for connections between this metric and the equivalent Wasserstein distance. For monographs on these topics, see for instance \cite{villani2003topics,ambrosio2008,Villani2009,Friesecke2025}.

\medskip

We assume a zero contact-angle condition $u_\xi \stackrel{\eqref{tfe-self-contact}}{=} 0$ at $\partial \{u > 0\}$ and obtain the first variation through integration by parts
\begin{align*}
D \EEE[u] [v] &= \int_{\{u > 0\}} \Big( u_\xi v_\xi + \frac 1 3 u^3 v - \frac 1 3 U_0^3 v + \frac{\xi^2}{10} v + U_0'' v\Big) \d \xi \\
&= - \int_{\{u > 0\}} \Big( u_{\xi\xi} + \frac 1 3 (U_0^3 - u^3) - \frac{\xi^2}{10} - U_0''\Big) v \, \d \xi.
\end{align*}
From \eqref{self-int-stable-2} it is immediate that $D \EEE[U] = 0$, that is, the self-similar solution $u = U$ is a critical point of $\EEE$. An elementary computation entails under the additional constraint $\int v \, \d \xi = 0$ that
\begin{align*}
D \EEE[u] [v] &= \int_{\{u > 0\}} \Big( u_{\xi\xi} + \frac 1 3 (U_0^3 - u^3) - \frac{\xi^2}{10} - U_0''\Big)_\xi (\partial_\xi^{-1} v) \, \d \xi \\
&= \int_{\{u > 0\}} u^{-1} \Big(\partial_\xi^{-1}\Big(u u_{\xi\xi\xi} - \frac{1}{4} (u^4)_\xi - \frac{\xi}{5} u\Big)_\xi\Big) (\partial_\xi^{-1} v) \, \d \xi = \langle \nabla_u \EE[u], v \rangle_u.
\end{align*}
This yields the gradient
\begin{equation}\label{grad-self}
\nabla_u \EEE[u] = \Big(u \Big(u_{\xi\xi\xi} - u^2 u_\xi - \frac{\xi}{5}\Big)\Big)_\xi.
\end{equation}
This entails that the thin-film equation \eqref{tfe-self} in self-similar variables is a gradient flow $u_s = - \nabla_u \EEE[u]$, where the gradient is taken with respect to the metric \eqref{inner}.

Note that at this stage the above considerations are formal. However, rather than using the above gradient-flow formulation, in what follows we instead pass to mass-Lagrangian coordinates. The governing partial differential equation (PDE) can then be formulated as a weighted $L^2$-gradient flow, where the weight is given by the self-similar profile $U$.

\subsection{Mass-Lagrangian formulation\label{sec-mass-lag}}
We introduce mass-Lagrangian coordinates (a convenient transform as the PDE \eqref{tfe-stable} is in divergence form and thus conserves mass) in conjunction with a logarithmic time scale and a suitable normalization that factors off the leading-order time asymptotics:
\begin{equation}\label{mass-lagrangian-stable}
\int_{Y(t,-\ell)}^{Y(t,x)} h(t,y) \, \d y = \int_{-\ell}^x U(\tilde x) \, \d \tilde x \quad \text{with} \quad \{h > 0\} = (Y(t,-\ell),Y(t,\ell)), \quad Z \coloneqq e^{-\frac s 5} Y, \quad \text{and} \quad s \coloneqq \ln (t+1).
\end{equation}
This transformation automatically conserves mass and we have $Z = x$ for the self-similar solution. From now on, we assume $Z_x \ge c > 0$ for some $c > 0$ (which we will justify further below). Using \eqref{tfe-stable} we obtain from \eqref{mass-lagrangian-stable},
\begin{equation}\label{mass-lagrangian-stable-2}
h Y_t - h h_{yyy} + \frac 1 4 (h^4)_y = 0, \qquad \partial_y = Y_x^{-1} \partial_x = e^{- \frac s 5} Z_x^{-1} \partial_x, \qquad \partial_t = e^{-s} \partial_s, \qquad h = U Y_x^{-1} = e^{- \frac s 5} U Z_x^{-1}.
\end{equation}
This results in
\begin{align*}
h Y_t &= e^{- \frac s 5} U Z_x^{-1} e^{-s} (e^{\frac s 5} Z)_s = e^{-s} U Z_x^{-1} (Z_s + \tfrac 1 5 Z), \\
- h h_{yyy} &= - e^{-s} U Z_x^{-1} (Z_x^{-1} \partial_x)^3 (U Z_x^{-1}), \\
h^3 h_y &= \frac{e^{-s}}{4} Z_x^{-1} \partial_x (U Z_x^{-1})^4 = e^{-s} (U Z_x^{-1})^3 Z_x^{-1} \partial_x (U Z_x^{-1}),
\end{align*}
so that we obtain the PDE
\begin{equation}\label{tfe-Z-stable}
\partial_s Z - \partial_Z^3 \Theta + \Theta^2 \partial_Z \Theta + \frac Z 5 = 0 \quad \text{in} \quad (0,\infty) \times (-\ell,\ell).
\end{equation}
where
\begin{equation}\label{def-theta-pz}
\partial_Z \coloneqq Z_x^{-1} \partial_x \qquad \text{and} \qquad \Theta \coloneqq U Z_x^{-1},
\end{equation}
and where in view of \eqref{grad-self} we have introduced convenient abbreviations ($\partial_Z$ and $\Theta$ are reminiscent of $\partial_\xi$ and $u$, respectively). Note that for $Z = x$, equation~\eqref{tfe-Z-stable} transforms into
\[
- U''' + U^2 U' + \frac x 5 = 0 \quad \text{in} \quad (-\ell,\ell),
\]
which is indeed \eqref{self-int-stable}. Hence, the self-similar solution $Z = x$ is a stationary solution to \eqref{tfe-Z-stable}, and we pass on to analyzing its stability in the sequel.

\subsection{Heuristics: gradient-flow formulation in mass-Lagrangian coordinates}
We transform the energy functional \eqref{energy} into mass-Lagrangian coordinates. For that observe
\begin{align*}
u \stackrel{\eqref{self-var-stable}}{=} e^{\frac s 5} h \;\, &\stackrel{\mathclap{\eqref{mass-lagrangian-stable-2}}}{=} \,\; U Z_x^{-1} \stackrel{\eqref{def-theta-pz}}{=} \Theta, \\
\xi \stackrel{\eqref{self-var-stable}}{=} e^{- \frac s 5} Y \;\, &\stackrel{\mathclap{\eqref{mass-lagrangian-stable}}}{=} \;\, Z, \\
\partial_\xi \stackrel{\eqref{self-var-stable}}{=} e^{\frac s 5} \partial_y \;\, &\stackrel{\mathclap{\eqref{mass-lagrangian-stable-2}}}{=} \;\, Z_x^{-1} \partial_x \stackrel{\eqref{def-theta-pz}}{=} \partial_Z, \\
\d\xi \stackrel{\eqref{self-var-stable}}{=} e^{-\frac s 5} \d y \;\, &\stackrel{\mathclap{\eqref{mass-lagrangian-stable-2}}}{=} \;\, Z_x \, \d x \eqqcolon \d Z,
\end{align*}
where \eqref{mass-lagrangian-stable} and \eqref{mass-lagrangian-stable-2} constitute the mass-Langrangian transform. Then
\begin{align*}
\EEE[u] &\stackrel{\eqref{energy}}{=} \int_{\{u > 0\}} \Big(\frac 1 2 u_\xi^2 + \frac{1}{12} u^4 - \frac 1 3 U_0^3 u + \frac{\xi^2}{10} u + U_0'' u\Big) \, \d \xi \\
&= \int_{-\ell}^\ell \Big(\frac 1 2 Z_x^{-1} (U Z_x^{-1})_x^2 + \frac{1}{12} U^4 Z_x^{-3} - \frac 1 3 U_0^3 U + \frac{Z^2}{10} U  + U_0'' U\Big) \, \d x \\
&= \int_{\{\Theta > 0\}} \Big(\frac 1 2 (\partial_Z \Theta)^2 + \frac{1}{12} \Theta^4 - \frac 1 3 U_0^3 \Theta + \frac{Z^2}{10} \Theta  + U_0'' \Theta\Big) \, \d Z.
\end{align*}
Discarding constant terms, this motivates the definition of the energy functional
\begin{equation}\label{energy-stable}
E[Z] \coloneqq \int_{-\ell}^\ell \Big(\frac 1 2 Z_x^{-1} (U Z_x^{-1})_x^2 + \frac{1}{12} U^4 Z_x^{-3} + \frac{Z^2}{10} U \Big) \,  \d x = \int_{\{\Theta > 0\}} \Big(\frac 1 2 (\partial_Z \Theta)^2 + \frac{1}{12} \Theta^4 + \frac{Z^2}{10} \Theta \Big) \,  \d Z.
\end{equation}
Provided that $V, Z \colon [-\ell,\ell] \to \R$ are sufficiently regular, and $Z_x \ge c > 0$ for some constant $c > 0$, we can compute the differential through integration by parts using $U = U' = 0$ at $x = \pm \ell$,
\begin{align*}
& D E[Z] [V] \\
& \quad = \; \int_{-\ell}^\ell \Big(- \frac 1 2 Z_x^{-2} V_x (U Z_x^{-1})_x^2 - Z_x^{-1} (U Z_x^{-1})_x (U Z_x^{-2} V_x)_x - \frac 1 4 (U Z_x^{-1})^4 V_x + \frac 1 5 Z U V \Big) \, \d x \\
& \quad = \; \int_{-\ell}^\ell \Big(\frac 1 2 (Z_x^{-2} (U Z_x^{-1})_x^2)_x V + U Z_x^{-2} (Z_x^{-1} (U Z_x^{-1})_x)_x V_x + (U Z_x^{-1})^3 (U Z_x^{-1})_x V + \frac 1 5 Z U V \Big) \, \d x \\
& \quad = \; \int_{-\ell}^\ell \Big(Z_x^{-1} (U Z_x^{-1})_x (Z_x^{-1} (U Z_x^{-1})_x)_x - (U Z_x^{-2} (Z_x^{-1} (U Z_x^{-1})_x)_x)_x + (U Z_x^{-1})^3 (U Z_x^{-1})_x + \frac 1 5 Z U \Big) V \, \d x \\
& \quad = \; \int_{-\ell}^\ell U \Big(- Z_x^{-1}(Z_x^{-1}(Z_x^{-1}(U Z_x^{-1})_x)_x)_x + U^2 Z_x^{-3} (U Z_x^{-1})_x + \frac Z 5\Big) V \, \d x \\
& \quad \stackrel{\mathclap{\eqref{def-theta-pz}}}{=} \; \int_{\{\Theta > 0\}} \Theta \Big(- \partial_Z^3 \Theta + \Theta^2 \partial_Z \Theta + \frac Z 5\Big) V \, \d Z,
\end{align*}
so that we obtain with $D E[Z] [V] = (\nabla_U E[Z], V)_U$. Here, the inner product (the metric) is given by the weighted $L^2$-inner product
\begin{equation}\label{inner-uvw}
(V, W)_U \coloneqq \int_{-\ell}^\ell U V W \, \d x \stackrel{\eqref{def-theta-pz}}{=} \int_{\{\Theta > 0\}} \Theta V W \, \d Z,
\end{equation}
and
\begin{equation}\label{grad-u-ez}
\nabla_U E[Z] = - \partial_Z^3 \Theta + \Theta^2 \partial_Z \Theta + \frac Z 5.
\end{equation}
We write $|V|_U \coloneqq \sqrt{( V, V )_U}$ for the induced norm. This entails that \eqref{tfe-Z-stable} can be formulated as a gradient flow
\begin{equation}\label{tfe-Z-stable-grad}
\partial_s Z + \nabla_U E[Z] = 0 \quad \text{in} \quad (0,\infty) \times (-\ell,\ell).
\end{equation}
Note that the formulation as a weighted $L^2$-gradient flow is in analogy with \cite[Theorem~2.18]{villani2003topics} in which the Wasserstein $L^2$-metric in one dimension is identified with the $L^2$-distance of the inverse cumulative distribution functions.

\subsection{Heuristics: lower bound on the Hessian \label{sec-hess-formal}}
In self-similar variables the energy functional reads
\[
\EEE[u] \stackrel{\eqref{energy}}{=} \int_{\{u > 0\}} \Big(\frac 1 2 u_\xi^2 + \frac{1}{12} u^4 - \frac 1 3 U_0^3 u + \frac{\xi^2}{10} u + U_0'' u\Big) \, \d \xi,
\]
so that
\[
D^2 \EEE[u] [v,v] = \int_{\{u > 0\}} (v_\xi^2 + u^2 v^2) \, \d \xi \ge 0,
\]
entailing convexity everywhere. We therefore investigate whether an analogous property also holds in the mass-Lagrangian formulation, in which the contact lines $Y(t,\pm\ell)$ and the solution close to them is better controlled. Therefore, consider \eqref{energy-stable}, where we assume $V, Z \colon (-\ell,\ell) \to \R$ to be sufficiently regular and $Z_x \ge c$ for some constant $c > 0$, we compute
\begin{align*}
D^2 \Big(\int_{-\ell}^\ell \frac{Z^2}{10} U\, \d x\Big) [V,V] &= \frac 1 5 \int_{-\ell}^\ell U V^2 \, \d x, \\
D^2 \Big(\int_{-\ell}^\ell \frac{1}{12} U^4 Z_x^{-3} \, \d x\Big) [V,V] &= \int_{-\ell}^\ell Z_x^{-1} (U Z_x^{-1})^4 V_x^2 \, \d x= \int_{-\ell}^\ell \Theta^4 (\partial_Z V)^2 \, \d Z.
\end{align*}
Furthermore,
\begin{align*}
D^2 \Big(\int_{-\ell}^\ell \frac 1 2 Z_x^{-1} (U Z_x^{-1})_x^2 \, \d x\Big) [V,V] &= \int_{-\ell}^\ell Z_x^{-3} V_x^2 (U Z_x^{-1})_x^2 \, \d x + 2 \int_{-\ell}^\ell Z_x^{-2} V_x (U Z_x^{-1})_x (U Z_x^{-2} V_x)_x \, \d x \\
&\phantom{=} + \int_{-\ell}^\ell Z_x^{-1} (U Z_x^{-2} V_x)_x^2 \, \d x + 2 \int_{-\ell}^\ell Z_x^{-1} (U Z_x^{-1})_x (U Z_x^{-3} V_x^2)_x \, \d x.
\end{align*}
After integration by parts, we obtain
\begin{align*}
& D^2 \Big(\int_{-\ell}^\ell \frac 1 2 Z_x^{-1} (U Z_x^{-1})_x^2 \, \d x\Big) [V,V] \\
& \quad = \int_{\{\Theta > 0\}}  \big((\partial_Z \Theta)^2 (\partial_Z V)^2 + 2 (\partial_Z \Theta) (\partial_Z V) \partial_Z (\Theta \partial_Z V) + (\partial_Z (\Theta \partial_Z V))^2 + 2 (\partial_Z \Theta) \partial_Z (\Theta (\partial_Z V)^2)\big) \, \d Z \\
& \quad = \int_{\{\Theta > 0\}} \big(6 (\partial_Z \Theta)^2 (\partial_Z V)^2 + \Theta^2 (\partial_Z^2 V)^2 + 8 \Theta (\partial_Z \Theta) (\partial_Z V) (\partial_Z^2 V)\big) \, \d Z \\
& \quad = \int_{\{\Theta > 0\}} \big((2 (\partial_Z \Theta)^2 - 4 \Theta \partial_Z^2 \Theta) (\partial_Z V)^2 + \Theta^2 (\partial_Z^2 V)^2\big) \, \d Z,
\end{align*}
so that
\[
D^2 E[Z][V,V] = \int_{\{\Theta > 0\}} \big(\tfrac 1 5 \Theta V^2 + (2 (\partial_Z \Theta)^2 - 4 \Theta (\partial_Z^2 \Theta) + \Theta^4) (\partial_Z V)^2 + \Theta^2 (\partial_Z^2 V)^2\big) \, \d Z.
\]
For $Z = x$ we then obtain
\[
D^2 E[x][V,V] = \int_{-\ell}^\ell \big(\tfrac 1 5 U V^2 + \Phi (V')^2 + U^2 (V'')^2\big) \, \d x,
\]
where
\begin{equation}\label{def-phi}
\Phi(x) \coloneqq  2 (U'(x))^2 - 4 U(x) U''(x) + (U(x))^4.
\end{equation}
We have $\frac U 5 > 0$ and $U^2 > 0$ in $(-\ell,\ell)$. We compute
\begin{equation}\label{der-phi}
\Phi' = 4 U (- U''' + U^2 U') \stackrel{\eqref{self-int-stable}}{=} - \frac 4 5 x U.
\end{equation}
Hence, $\Phi' < 0$ for $x > 0$ and $\Phi = 0$ at $x = \ell$, which entails by symmetry $\Phi(x) = \Phi(-x)$ that $\Phi > 0$ in $(-\ell,\ell)$. This entails 
the spectral gap (lower bound on the Hessian)
\begin{equation}\label{lower-hessian-z=x}
D^2 E[x][V,V] \ge \frac 1 5 |V|_U^2,
\end{equation}
which is crucial to prove convergence to the self-similar profile.

\medskip

In the general case, we get
\begin{align*}
\partial_Z (2 (\partial_Z \Theta)^2 - 4 \Theta (\partial_Z^2 \Theta) + \Theta^4) &= - 4 \Theta (\partial_Z^3 \Theta) + 4 \Theta^3 \partial_Z\Theta \stackrel{\eqref{grad-u-ez}}{=} 4 \Theta \big(\nabla_U E[Z] - \tfrac Z 5\big) \\
&= - \frac 4 5 x \Theta + 4 \Theta \big(\nabla_U E[Z] - \tfrac 1 5 (Z-x)\big).
\end{align*}
We recognize that formally
\[
\Theta = U (1+o(1)) \qquad \text{and} \qquad x^{-1} \big(\nabla_U E[Z] - \tfrac 1 5 (Z-x)\big) = o(1).
\]
It remains to verify this rigorously. We note that $\partial_Z V = 1 - Z_x^{-1}$ for $V = Z-x$ and $\Theta = U (1- (1-Z_x^{-1}))$. Hence, in order to estimate the remainder terms, it appears necessary to control $\|1 - Z_x^{-1}\|_{C^0([0,\infty) \times [-\ell,\ell])}$. Elementary scaling considerations close to the boundaries $x = \pm \ell$ (consider \eqref{energy-stable} and \eqref{inner-uvw}) imply that such strong control cannot come from purely energetic considerations only, so that we opt for a more refined linear analysis. Therefore, consider the Hessian $\cL$, which, by definition, is given by the idendity
\begin{equation}\label{def-hessian}
(V, \cL V)_ U = D^2 E[x][V,V] \quad \text{for all } V \in C^4([-\ell,\ell]).
\end{equation}
Thus we get
\begin{equation}\label{formula-l}
\cL V = U^{-1} \partial_x^2 (U^2 \partial_x^2 V) - U^{-1} \partial_x (\Phi \partial_x V) + \frac V 5,
\end{equation}
and obtain formal coercivity
\begin{equation}\label{coercivity-l}
(V, \cL V)_U = \int_{-\ell}^\ell \big(\tfrac U 5 V^2 + \Phi (\partial_x V)^2 + U^2 (\partial_x^2 V)^2\big) \, \d x \ge \frac 1 5 |V|_U^2 \quad \text{for all } V \in C^4([-\ell,\ell]).
\end{equation}
We set
\begin{equation}\label{def-v}
V \coloneqq Z-x
\end{equation}
and formulate the PDE \eqref{tfe-Z-stable-grad} as
\begin{equation}\label{pde-nonlinear}
\partial_s V + \cL V = N[V]
\end{equation}
where the nonlinearity is given by
\begin{equation}\label{nonlinearity}
N[V] \coloneqq \cL V + \nabla_U E[x] - \nabla_U E[x+V] = \cL V - \nabla_U E[x+V],
\end{equation}
because $\nabla_U E[x] = 0$. Since $\cL$ is symmetric with respect to $(\cdot,\cdot)_U$, we are in a position to derive higher-order estimates (test with $V$ or $\cL V$ in $(\cdot,\cdot)_U$, or apply $\cL$ to the equation). Thus we obtain control of $\|V_x\|_{C^0([0,\infty) \times [-\ell,\ell])}$ further below.

\subsection{A formal argument for stability\label{sec-stability-formal}}
Using the gradient-flow formulation \eqref{tfe-Z-stable-grad}, the fact that $\nabla_U E[x] = 0$, i.e., the self-similar solution $Z = x$ is a stationary point of the energy $E$, and the bound \eqref{lower-hessian-z=x} on the second variational derivative, we can provide a formal argument indicating that the self-similar solution is stable:
\begin{align*}
\frac{\d}{\d s} |Z-x|_U^2 \; &= \; 2 \langle Z-x, Z_s \rangle_U \stackrel{\eqref{tfe-Z-stable-grad}}{=} - 2 \langle Z-x, \nabla_U E[Z] \rangle_U = - 2 (D E[Z]) [Z-x] \\
&= \; - 2 \int_0^1 (D^2 E[r Z + (1-r) x])[Z-x,Z-x] \, \d r \stackrel{\eqref{def-hessian}}{=} - 2 (Z-x, \cL (Z-x))_U + o(|Z-x|_U^2) \\
&\stackrel{\mathclap{\eqref{coercivity-l}}}{\le} \; - \frac 2 5 (1+o(1)) |Z-x|_U^2.
\end{align*}
Gr\"onwall's lemma entails
\[
|Z(s)-x|_U \le |Z^{(0)}-x|_U \, e^{- \frac 1 5 (1+o(1)) s},
\]
that is, the self-similar solution $Z = x$ is exponentially stable in the logarithmic time $s \stackrel{\eqref{mass-lagrangian-stable}}{=} \ln(t+1)$ with respect to the norm $|\cdot|_U$. The above argumentation requires proximity to the self-similar solution $Z = x$, for which the lower bounds \eqref{lower-hessian-z=x} and equivalently \eqref{coercivity-l} apply. This will be made rigorous in the sequel.

\subsection{Formal estimates on the Hessian: refined linear analysis and functional-analytic setting\label{sec-formal-funcana}}
Consider the linear equation
\begin{equation}\label{pde-linear}
\partial_s V + \cL V = F,
\end{equation}
where $V \colon (0,\infty) \times (-\ell,\ell) \to \R$ and the right-hand side $F \colon (0,\infty) \times (-\ell,\ell) \to \R$ are sufficiently regular. Here, we derive suitable estimates on a formal level and make the arguments rigorous in what follows (see \S\ref{sec-linear}).

\medskip

We test \eqref{pde-linear} with $V$ in the inner product $(\cdot,\cdot)_U$ and obtain with \eqref{coercivity-l}
\begin{equation}\label{pde-tested}
\frac 1 2 \frac{\d}{\d s} |V|_U^2 + \int_{-\ell}^\ell \big(\tfrac U 5 V^2 + \Phi (\partial_x V)^2 + U^2 (\partial_x^2 V)^2\big) \, \d x = (F,V)_U.
\end{equation}
For each $k\in\mathbb N$, we define the inner product and the corresponding norm as follows:
\begin{subequations}\label{sobolev-hk}
\begin{equation}\label{weighted-inner}
(V,W)_k \coloneqq \sum_{j = 0}^k \int_{-\ell}^\ell (\ell^2-x^2)^{j+2} (\partial_x^j V) (\partial_x^j W) \, \d x, \qquad |V|_k \coloneqq \sqrt{(V,V)_k}, \quad \text{where} \quad V, W \in C^k([-\ell,\ell]).
\end{equation}
Note that by interpolation (cf.~Lemma~\ref{lem-interp} below), the norm $|V|_k$ is up to a $(k,\ell)$-dependent constant equivalent to the norm $|V|_{k,\bullet}$, where
\[
|V|_{k,\bullet}^2 \coloneqq \int_{-\ell}^\ell (\ell^2-x^2)^{2} V^2 \, \d x +\int_{-\ell}^\ell (\ell^2-x^2)^{k+2} (\partial_x^k V)^2 \, \d x.
\]
We define the weighted Sobolev spaces
\begin{equation}\label{weighted-sobolev}
\cH^k \coloneqq \text{closure of } C^k([-\ell,\ell]) \text{ with respect to } |\cdot|_k,
\end{equation}
\end{subequations}
and have with $U \sim (\ell^2-x^2)^2$ and $\Phi \sim (\ell^2-x^2)^3$ (cf.~Lemma~\ref{lem-bounds-uphi} below)
\begin{equation}\label{weak-estimate}
\|V\|_{L^\infty(0,\infty;\cH^0)} + \|V\|_{L^2(0,\infty;\cH^2)} \lesssim |V(0)|_0 + \|F\|_{L^2(0,\infty;\cH^0)}.
\end{equation}
Next, we test \eqref{pde-linear} with $\cL V$ in $(\cdot,\cdot)_U$ and obtain with analogous arguments on noting that with \eqref{coercivity-l},
\begin{align*}
(\cL V, \partial_s V)_U &= \frac 1 2 \frac{\d}{\d s} \int_{-\ell}^\ell \big(\tfrac U 5 V^2 + \Phi (\partial_x V)^2 + U^2 (\partial_x^2 V)^2\big) \, \d x, \\
\int_0^s (F,\cL V)_U \, \d s' &\le \frac \eps 2 \int_0^s |F|_U^2 \, \d s' + \frac 2 \eps \int_0^s |\cL V|_U^2 \, \d s' \quad \text{for} \quad \eps > 0,
\end{align*}
that
\[
\|V\|_{L^\infty(0,\infty;\cH^2)} + \|\cL V\|_{L^2(0,\infty;\cH^0)} \lesssim |V(0)|_2^2 + \|F\|_{L^2(0,\infty;\cH^0)}.
\]
With help of elliptic regularity of the linear operator $\cL$ (cf.~Lemma~\ref{lem-elliptic} below) we obtain
\[
\|V\|_{L^\infty(0,\infty;\cH^2)} + \|V\|_{L^2(0,\infty;\cH^4)} \lesssim |V(0)|_2^2 + \|F\|_{L^2(0,\infty;\cH^0)}.
\]
In conjunction with \eqref{pde-linear} to obtain control on $\partial_s V$, this entails the maximal-regularity estimate
\begin{equation}\label{strong-estimate}
\|\partial_s V\|_{L^2(0,\infty;\cH^0)} + \|V\|_{C^0([0,\infty);\cH^2)} + \|V\|_{L^2(0,\infty;\cH^4)} \lesssim |V(0)|_2^2 + \|F\|_{L^2(0,\infty;\cH^0)}.
\end{equation}
Further using the weak estimate \eqref{weak-estimate} for $\cL V$ instead of $V$ (and estimating the right-hand side more carefully) gives
\[
\|\cL V\|_{L^\infty(0,\infty;\cH^0)} + \|\cL V\|_{L^2(0,\infty;\cH^2)} \lesssim |\cL V(0)|_0 + \|\cL F\|_{L^2(0,\infty;\cH^0)},
\]
and elliptic regularity for $\cL$ (cf.~Lemma~\ref{lem-elliptic} below) in conjunction with \eqref{pde-linear} yields
\[
\|\partial_s V\|_{L^2(0,\infty;\cH^2)} + \|V\|_{C^0([0,\infty);\cH^4)} + \|V\|_{L^2(0,\infty;\cH^6)} \lesssim |V(0)|_4^2 + \|F\|_{L^2(0,\infty;\cH^2)}.
\]
Iterating the arguments, we obtain
\begin{equation}\label{maxreg-general}
\|\partial_s V\|_{L^2(0,\infty;\cH^{2k})} + \|V\|_{C^0([0,\infty);\cH^{2k+2})} + \|V\|_{L^2(0,\infty;\cH^{2k+4})} \lesssim |V(0)|_{2k+2} + \|F\|_{L^2(0,\infty;\cH^{2k})}
\end{equation}
for all integers $k \ge 0$. This maximal-regularity estimate \eqref{maxreg-general} in a time-weighted form will be proved further below in Proposition~\ref{prop-maxreg}. It will turn out to be the main ingredient for proving well-posedness of the nonlinear evolution \eqref{pde-nonlinear}. The main step for estimating the nonlinearity is then to have control on $\|V_x\|_{C^0([0,\infty) \times [-\ell,\ell])}$, which one anticipates by scaling arguments for $k \ge 2$, see Lemma~\ref{lem-c0} below.

\section{Main results and discussion\label{sec-main}}
\subsection{Main results}
Consider the nonlinear Cauchy problem (cf.~\eqref{pde-nonlinear})
\begin{subequations}\label{nonlinear-cauchy}
\begin{align}
\partial_s V + \cL V &= N[V] && \text{for} \quad s > 0, \label{pde-cauchy} \\
V &= V^{(0)} && \text{at} \quad s = 0,
\end{align}
\end{subequations}
where $V^{(0)} \colon (0,\infty) \to \R$ is given. We have the following well-posedness result:

\begin{theorem}[Well-posedness]\label{th-well}
For $V^{(0)} \in \cH^6$ such that $|V|_0 \ll_\ell 1$, there exists exactly one solution
\[
V \in H^1(0,\infty;\cH^4) \cap L^2(0,\infty;\cH^8) \cap C^0([0,\infty); \cH^6),
\]
to the nonlinear Cauchy problem \eqref{nonlinear-cauchy}. Additionally, this solution satisfies
\begin{subequations}\label{regularity-weighted-2}
\begin{align}
(s \mapsto s^{\frac{k-2}{2}} (\partial_s^m V)(s+s_{k,\ell})) &\in L^2(0,\infty;\cH^{2k+4-4m}) && \text{for} \quad k \in \N, \quad k \ge 2, \quad m \in \{0,1\}, \\
(s \mapsto s^{\frac{k-2}{2}} V(s+s_{k,\ell})) &\in C^0([0,\infty);\cH^{2k+2}) && \text{for} \quad k \in \N, \quad k \ge 2,
\end{align}
\end{subequations}
where $s_{2,\ell} = s_{3,\ell} = 0$ and $s_{k,\ell} \gg_{k,\ell} 1$ for $k \ge 4$. Furthermore, we have the {\it \`a-priori} estimate
\begin{equation}\label{a-priori-nonlin}
\sum_{k' = 2}^k \Big(\sup_{s \ge 0} s^{k'-2} |V(s+s_{k',\ell})|_{2k'+2}^2 + \int_0^\infty s^{k'-2} \big(|(\partial_s V)(s+s_{k',\ell})|_{2k'}^2 + |V(s+s_{k',\ell})|_{2k'+4}^2\big) \, \d s\Big) \lesssim_{k,\ell} |V^{(0)}|_6^2
\end{equation}
for any $k \in \N$ with $k \ge 2$.
\end{theorem}

Based on the well-posedness proved in Theorem~\ref{th-well}, we can derive the following stability result:
\begin{theorem}[Stability]\label{th-stability}
Suppose $V^{(0)} \in \cH^6$ is such that $|V^{(0)}|_6 \ll 1$ and let
\[
V \in H^1(0,\infty;\cH^4) \cap L^2(0,\infty;\cH^8) \cap C^0([0,\infty); \cH^6),
\]
be the unique solution to problem \eqref{nonlinear-cauchy} given by Theorem~\ref{th-well} and satisfying
\begin{align*}
(s \mapsto s^{\frac{k-2}{2}} (\partial_s^m V)(s+s_{k,\ell})) &\in L^2(0,\infty;\cH^{2k+4-4m}) && \text{for} \quad k \in \N, \quad k \ge 2, \quad m \in \{0,1\}, \\
(s \mapsto s^{\frac{k-2}{2}} V(s+s_{k,\ell})) &\in C^0([0,\infty);\cH^{2k+2}) && \text{for} \quad k \in \N, \quad k \ge 2,
\end{align*}
where $s_{2,\ell} = s_{3,\ell} = 0$ and $s_{k,\ell} \gg_{k,\ell} 1$ for $k \ge 4$. Then it holds
\begin{equation}\label{est-decay}
|\partial_s^m V(s)|_k \lesssim_{k,\ell} e^{- \frac s 5} \quad \text{for} \quad s \gg_{k,\ell} 1 \quad \text{and} \quad m \in \{0,1\}.
\end{equation}
\end{theorem}
\subsection{Asymptotics in the original variables}
We first transform back into the original variables by noticing that
\[
Z \stackrel{\eqref{def-v}}{=} x + V \quad \Rightarrow \quad Y \stackrel{\eqref{mass-lagrangian-stable}}{=} e^{\frac s 5} Z = e^{\frac s 5} (x+V).
\]
Thus we find with Theorem~\ref{th-stability}
\[
Y_t \stackrel{\eqref{mass-lagrangian-stable-2}}{=} e^{-s} Y_s = e^{- \frac 4 5 s} \big(\tfrac x 5 + \big(\partial_s+\tfrac 1 5\big) V\big) \stackrel{\eqref{est-decay}, \eqref{c0_embed}}{=} e^{- \frac 4 5 s} \big(\tfrac x 5 + O(e^{- \frac s 5})\big) \stackrel{\eqref{mass-lagrangian-stable}}{=} t^{-\frac 4 5} \big(\tfrac x 5 + O(t^{-\frac 1 5})\big) \quad \text{as} \quad t \to \infty,
\]
where we have used Lemma~\ref{lem-c0} below to obtain a uniform bound for the symbol $O(\cdot)$. Hence, we obtain for the velocities at the contact lines $x = \pm \ell$,
\[
Y_t(t,\pm\ell) = t^{-\frac 4 5} \big(\pm \tfrac \ell 5 + O(t^{-\frac 1 5})\big) \quad \text{as} \quad t \to \infty.
\]
The same propagation speed was found by Bernis in \cite{bernis1996finite} (see \cite{bernis1996finite2} for the case of weak slippage) for the thin-film equation without gravity, matching the speed of the source-type solution \cite{smyth1988high,bernis1992source}.

\medskip

We further analyze the asymptotics of $h$ by noticing that
\[
h \stackrel{\eqref{mass-lagrangian-stable-2}}{=} e^{- \frac s 5} U Z_x^{-1} \stackrel{\eqref{def-v}}{=} e^{- \frac s 5} U (1+V_x)^{-1} \stackrel{\eqref{est-decay}, \eqref{c0_embed}}{=} e^{- \frac s 5} U (1+O(e^{-\frac s 5})) \stackrel{\eqref{mass-lagrangian-stable}}{=} t^{- \frac 1 5} U (1+O(t^{-\frac 1 5})) \quad \text{as} \quad t \to \infty,
\]
where Lemma~\ref{lem-c0} below entails a uniform bound for the symbol $O(\cdot)$. We note that in \cite[\S2]{gnann2015well} similar findings were made, but that the rate of convergence is improved compared to the situation at hand (convergence with rate $\sim t^{-1+0}$). The difference is subtle: In our setting, the symmetric form of the linearization including coercivity hinges on a lower bound of the Hessian in a suitable gradient-flow formulation. In \cite{gnann2015well}, the linearization is given in terms of a second-order polynomial of the second-order Gegenbauer differential operator \cite{suetin2001ultraspherical}. This structure is stable and leads to a changed weight on reformulating the equation in terms of $u \coloneqq Z_x^{-1}-1$, a set of coordinates for which translations of the profile are not seen, thus leading to the faster convergence rate. In our setting this symmetric form appears to be lost for $u \coloneqq Z_x^{-1}-1$, at least in terms of a simple weighted inner product, and thus the translation limits the rate of convergence.

\subsection{Discussion}
We now give a brief exposition of the existing literature regarding convergence results. Note that this list is non-exhaustive and aims at highlighting the novelty of our result relative to what is known:

\subsubsection*{Porous-medium equation}
We mention the convergence result of Carrillo and Toscani \cite{carrillo2000asymptotic} for convergence in the $L^1$-norm towards the Barenblatt-Pattle solution also in higher dimensions using entropy methods. Furthermore, Otto in \cite{Otto2001} has proved convergence to the Barenblatt-Pattle solution relying on geometric considerations in the Wasserstein space. Well-posedness of perturbations around the Barenblatt-Pattle solution has been found by Koch in \cite{koch1999}, while the higher asymptotics have been analyzed in the linear case in one spatial dimension by Bernoff and Witelski in \cite{witelski1998self}, and in higher dimensions by Denzler and McCann in \cite{denzler2008nonlinear}. These higher asymptotics have been lifted to the nonlinear setting by studying invariant manifolds in \cite{seis2015invariant} by Seis. This analysis in turn relies on a corresponding analysis for the fast-diffusion equation by Denzler, Koch, and McCann in \cite{denzler2015higher}.

\subsubsection*{Cahn-Hilliard equation}
For the Cahn-Hilliard equation we mention for instance the nonlinear stability analysis of the kink solution (hyperbolic tangent) by Howard in \cite{howard2007asymptotic} relying on the analysis of the associated Evans function that is lifted to the nonlinear setting. Using gradient-flow techniques and energetic considerations, Otto and Westdickenberg in \cite{otto2014relaxation} have proved convergence to this kink solution with optimal relaxation rates without going into tedious asymptotics like Howard. The latter analysis also provides the possibility to be lifted to special solutions that are not in explicit form, see a respective follow-up work of Otto, Scholtes, and Westdickenberg in \cite{otto-cahnhilliard-2019}. The case of the bump was treated by Biesenbach, Schubert, and Westdickenberg in \cite{Biesenbach04032022}.

\subsubsection*{Thin-film equation with linear mobility and without gravity}
In this case, we mention the work of Bernoff and Witelski \cite{bernoff2002linear}, in which for the thin-film equation
\begin{equation}\label{tfe-higher}
h_t + \nabla \cdot (h \nabla \Delta h) = 0 \quad \text{in} \quad \{h > 0\}, \qquad \text{subject to} \quad \nabla h = 0 \quad \text{on} \quad \partial\{h > 0\},
\end{equation}
the linear stability of self-similar solutions in one spatial dimension is analyzed and an (higher-order) asymptotics towards the self-similar solution (a fourth-order symmetric polynomial with zero contact angle, i.e., up to scaling and translation $(1-|x|^2)^2$), first found by Smyth and Hill in \cite{smyth1988high}, are obtained. Carrillo and Toscani in \cite{carrillo2002long} have proved convergence of arbitrary initial data in the $L^1$-norm with algebraic rate $\sim t^{-\frac 1 5}$ in one spatial dimension, relying on the analogous analysis for the porous-medium equation (convergence to the Barenblatt-Pattle profile) in \cite{carrillo2000asymptotic}. This result was subsequently upgraded by Carlen and Ulusoy in \cite{carlen2007asymptotic} to convergence in the $H^1$-norm. The higher-dimensional case was treated by Matthes, McCann, and Savar\'e in \cite{matthes2009family} for a family of gradient flows including the thin-film equation \eqref{tfe-higher} with linear mobility and the Derrida-Lebowitz-Speer-Spohn (DLSS) equation \cite{derrida1991dynamics}. Convergence of moments of the one-dimensional variant of \eqref{tfe-higher} towards the self-similar solution is obtained by Carlen and Ulusoy in \cite{carlen2014localization}.

\medskip

More recently, focus has been placed on more refined linear analyses leading to nonlinear stability results and higher asymptotics. In this vein, the first author of this paper has proved in \cite{gnann2015well} well-posedness of perturbations and convergence to the Smyth-Hill solution in arbitrary (Sobolev) regularity and one spatial dimension with rate $\sim t^{-1+0}$, as the chosen set of coordinates allows to remove translations (limiting the above results to rates $\sim t^{-\frac 1 5}$). This relies on a corresponding well-posedness and stability result of the half parabola $(y_+)^2$ due to Bringmann, Giacomelli, Kn\"upfer, and Otto in \cite{bringmann2016corrigendum,giacomelli2008smooth}. McCann and Seis in \cite{mccann2015spectrum} analyze higher asymptotics to the linearization of \eqref{tfe-higher} also in higher dimensions, and Seis in \cite{seis2018thin} proves well-posedness of perturbations of self-similar solutions in higher dimensions including convergence, in analogy to what has been proved for the porous-medium equation in \cite{koch1999,Kienzler2016} by Kienzler and Koch and for the thin-film equation in \cite{John2015} by John. Higher asymptotics have been obtained by Seis and Winkler in \cite{seis2024invariant} by studying suitable invariant manifolds, the analysis being analogous to the porous-medium equation in \cite{seis2015invariant} by Seis.

\medskip

Notably, the literature in case of nonzero contact angle at the contact line $\partial\{h > 0\}$ is less exhaustive but we mention that convergence of compactly-supported solutions to the inverted parabola $(1-x^2)_+$ in one dimension was found by Majdoub, Masmoudi, and Tayachi in \cite{MajdoubMasmoudiTayachi2021} relying on the corresponding analysis of Esselborn in \cite{esselborn2016relaxation}, studying stability of the linear profile $y_+$ using energetic considerations and relying on the gradient-flow structure of the equation, which is analogous to the work of Otto and Westdickenberg for the Cahn-Hilliard equation \cite{otto2014relaxation}. See also Kn\"upfer and Masmoudi in \cite{KnuepferMasmoudi2013,KnuepferMasmoudi2015} for a well-posedness and stability analysis to the thin-film equation around the linear profile $y_+$ including a rigorous lubrication approximation coming from Darcy dynamics in the Hele-Shaw cell.

\subsubsection*{A perturbative results to the thin-film equation with linear mobility and second-order porous-medium type term}
Parsch in \cite{parsch2025globalexponentialstabilitystationary} considered the PDE
\begin{equation}\label{parsch}
u_s = -(u u_{\xi\xi\xi})_\xi + \lambda (\xi u)_\xi + \eps (u h'(u)_\xi)_\xi.
\end{equation}
Note that \eqref{parsch} reduces to \eqref{tfe-self} (reformulation of \eqref{tfe-stable} in self-similar variables) on setting $\lambda = \frac 1 5$, $\eps = 1$, and $h(u) = \frac{u^4}{12}$. Further note that $h$ must be cut off to suitable growth for initial data in $H^1(\R) \hookrightarrow L^\infty(\R)$ provided the solution stays bounded to meet Parsch's growth constraint \cite[(9)]{parsch2025globalexponentialstabilitystationary} for $h$. However, boundedness for $s > 0$ is not proved so far (for $H^1$-convergence in a related case, which would imply this, see for instance \cite{matthes2009family}).

Parsch proves existence of weak solutions and for $0 \le \eps < \bar\eps$, with $\bar\eps > 0$ sufficiently small, existence of a steady state (the self-similar solution, i.e., the perturbed Smyth-Hill profile \cite{smyth1988high}) and $L^1$-convergence to the steady state with exponential rate $\sim e^{-(\lambda - C\eps) s}$ for some $C < \infty$. The result \cite{parsch2025globalexponentialstabilitystationary} therefore does not quite cover the situation treated in this paper, but may be lifted in the long run to allow for a non-perturbative analysis by proving that one can choose $\eps = 1$ in \eqref{parsch}, a non-perturbed decay rate $\sim e^{-\lambda s}$, as well as boundedness to remove the growth constraint \cite[(9)]{parsch2025globalexponentialstabilitystationary}.

\subsubsection*{Summary}
In conclusion, the situation at hand covers for the first time a non-perturbative/truncated well-posedness result, Theorem~\ref{th-well}, and a convergence result, Theorem~\ref{th-stability}, for a thin-film equation in which the special solution, around which one perturbs, is not given explicitly, but is only described as a solution to an ODE, i.e., \eqref{self-int-stable}. We leave the endeavor to develop a corresponding gradient-flow perspective with optimal convergence rates, such as in \cite{otto-cahnhilliard-2019} for the Cahn-Hilliard equation, to future work. Furthermore, as in \cite{bernoff2002linear,mccann2015spectrum,seis2024invariant} one could study the higher asymptotics to \eqref{nonlinear-cauchy} by a more refined analysis of the linear spectrum in conjunction with analyzing corresponding invariant manifolds.

\subsection{Outline}
In what follows, we prove Theorems~\ref{th-well} and \ref{th-stability}. Therefore, the linear theory is developed in \S\ref{sec-linear}. This is split into proving coercivity and elliptic regularity in \S\ref{sec-elliptic}, which is the basis for deriving suitable resolvent estimates in \S\ref{sec-resolvent}. The latter would imply maximal regularity using semi-group theory, but in the $L^2$-setting at hand it is also possible to directly derive by elementary arguments maximal regularity using a time-stepping procedure in \S\ref{sec-parabolic}. The nonlinear equation is treated in \S\ref{sec-nonlinear}, that is, suitable estimates of the nonlinearity are proved in \S\ref{sec-nonlinearity}, based on which, together with the linear estimates derived in \S\ref{sec-parabolic}, well-posedness in form of Theorem~\ref{th-well} is proved in \S\ref{sec-well}. The paper is concluded with the proof of stability in form of Theorem~\ref{th-stability} in \S\ref{sec-stability}.

\subsection{Notation and conventions}
We write $A \lesssim_P B$ whenever a constant $C < \infty$ only depending on the set of parameters $P$ exists such that $A \le C B$.

\medskip

For $\alpha \in \R$ we denote by $\lfloor \alpha \rfloor \coloneqq \max\{m \in \Z\colon m \le\alpha\}$ the integer part of $\alpha$.

\medskip

We write $\N = \{1,2,3,\ldots\}$ and $\N_0 = \{0,1,2,3,\ldots\}$.

\medskip

For a set $\Omega \in \R^d$, a Banach space $X$, and $1 \le p \le \infty$, we write $L^p(\Omega;X)$ for the Bochner space of $p$-integrable functions $\Omega \to X$. For $k \in \N_0$ we write $C^k(\Omega;X)$ for the space of $k$-times continuously differentiable functions $\Omega \to X$.

\section{The linear degenerate-parabolic equation}\label{sec-linear}
In this section, we prove existence and uniqueness of solutions with suitable {\it \`a-priori} estimates of the linear equation
\begin{equation}\label{lin-pde}
\partial_t V + \cL V = F \quad \text{in} \quad (0,\infty) \times (-\ell,\ell),
\end{equation}
where $V$ is the unknown for a given right-hand side $F$, and where $\cL$ is given by
\[
\cL V \stackrel{\eqref{formula-l}}{=} U^{-1} \partial_x^2 (U^2 \partial_x^2 V) - U^{-1} \partial_x (\Phi \partial_x V) + \frac V 5 \quad \text{with} \quad \Phi \stackrel{\eqref{def-phi}}{=} 2 (U')^2 - 4 U U'' + U^4,
\]
for some $V \in C^\infty([-\ell,\ell])$. With help of $\Phi' \stackrel{\eqref{der-phi}}{=} - \frac 4 5 x U$ we infer
\begin{equation}\label{formula-l-2}
\cL V = U \partial_x^4 V + 4 U' \partial_x^3 V + (6 U''-U^3) \partial_x^2 V + \frac 4 5 x \partial_x V + \frac{V}{5}.
\end{equation}
\subsection{Coercivity and elliptic regularity}\label{sec-elliptic}
Using Leibnitz's rule, we can deduce the following representations for derivatives:
\begin{lemma}\label{lem-der-l}
For $k \in \N_0$ and $V \in C^\infty([-\ell,\ell])$ it holds
\begin{equation}\label{derivatives-l}
\partial_x^k \cL V = U \partial_x^{k+4} V + (k+4) U' \partial_x^{k+3} V + \frac 1 2 \big((k+3)(k+4) U'' - 2 U^3\big) \partial_x^{k+2} V + \sum_{j = \min\{k,2\}}^{k+1} A_{k,j} \partial_x^j V,
\end{equation}
where $A_{k,j} \in C^\infty([-\ell,\ell])$ for $j = \min\{k,2\},\ldots,k+1$. Furthermore,
\begin{equation}\label{derivatives-l-2}
\partial_x^k \cL V = U^{-\frac{k+2}{2}} \partial_x^2 (U^{\frac{k+4}{2}} \partial_x^{k+2} V) + (8 U)^{-1} \big(- (k+2)(k+4) \Phi + k (k+6) U^4\big) \partial_x^{k+2} V + \sum_{j = \min\{k,2\}}^{k+1} A_{k,j} \partial_x^j V.
\end{equation}
\end{lemma}
We will justify below that the symmetrized term $U^{-\frac{k+2}{2}} \partial_x^2 (U^{\frac{k+4}{2}} \partial_x^{k+2} V)$ is the scaling-wise leading part of $\partial_x^k \cL V$ close to the boundaries $x = \pm \ell$, which is why we refer to it as the \emph{principal part} of $\partial_x^k \cL V$.
\begin{proof}[Proof of Lemma~\ref{lem-der-l}]
Differentiating \eqref{formula-l-2} using the ODE \eqref{self-int-stable} for $U$, we find
\begin{align*}
\partial_x \cL V &= U \partial_x^5 V + 5 U' \partial_x^4 V + (10 U'' - U^3) \partial_x^3 V + (3 U^2 U' + 2x) \partial_x^2 V + \partial_x V, \\
\partial_x^2 \cL V &= U \partial_x^6 V + 6 U' \partial_x^5 V + (15 U'' - U^3) \partial_x^4 V + (10 U^2 U' + 4 x) \partial_x^3 V + 3 (U^2 U'' + 2 U (U')^2 + 1) \partial_x^2 V.
\end{align*}
The formula \eqref{derivatives-l} follows by induction using the ODE \eqref{self-int-stable}. With help of \eqref{def-phi} we then arrive at \eqref{derivatives-l-2}.
\end{proof}

We recall the scale of weighted Sobolev spaces in \eqref{sobolev-hk}, that is,
\begin{align}
 \cH^k \coloneqq \; & \text{closure of } C^\infty([-\ell,\ell]) \text{ with respect to } |\cdot|_k, \nonumber\\
&\text{where } |V|_k^2 \coloneqq \sqrt{(V,V)_k} \text{ with } (V,W)_k \coloneqq \sum_{j = 0}^k \int_{-\ell}^\ell (\ell^2-x^2)^{j+2} (\partial_x^j V) (\partial_x^j W) \, \d x. \nonumber
\end{align}
Furthermore, we introduce the semi-norms
\begin{equation}\label{semi-norm}
[V]_k^2 \coloneqq \int_{-\ell}^\ell (\ell^2-x^2)^{k+2} (\partial_x^k V)^2 \, \d x \quad \text{for } V \in \cH^k.
\end{equation}
In the sequel, the following lemma is convenient for adapting the weight in the norm:
\begin{lemma}\label{lem-bounds-uphi}
There exists $0 < c \le C < \infty$ depending on $\ell$ such that
\begin{subequations}\label{est-u-ell^2-x^2}
\begin{align}
c (\ell^2-x^2)^3 &\le \Phi(x) \le C (\ell^2-x^2)^3 && \text{for all } x \in (-\ell,\ell), \label{bounds-phi} \\
c (\ell^2-x^2)^2 &\le |\Phi'(x)| \le C (\ell^2-x^2)^2 && \text{for all } x \in (-\ell,\ell), \label{bounds-phip} \\
c (\ell^2-x^2) &\le |\Phi''(x)| \le C (\ell^2-x^2) && \text{for all } x \in (-\ell,\ell), \label{bounds-phipp} \\
c (\ell^2-x^2)^2 &\le U(x) \le C (\ell^2-x^2)^2 && \text{for all } x \in (-\ell,\ell), \label{bounds-u} \\
c (\ell^2-x^2) &\le |U'(x)| \le C (\ell^2-x^2) && \text{for all } x \in (-\ell,\ell). \label{bounds-up}
\end{align}
\end{subequations}
\end{lemma}
\begin{proof}
From \cite[Lemma~2.2]{beretta1997selfsimilar} we know that $U > 0$ in $(-\ell,\ell)$ and from \cite[Theorem~1.3~(a)]{bernis1992source} and \cite[Theorem~5.1]{beretta1997selfsimilar} $U = U' = 0$ and $U'' > 0$ at $x = \pm \ell$. This already entails \eqref{bounds-u} and \eqref{bounds-up} by continuity. Because of
\[
\Phi \stackrel{\eqref{def-phi}}{=} 2 (U')^2 - 4 U U'' + U^4 \qquad \text{and} \qquad \Phi' \stackrel{\eqref{der-phi}}{=} - \frac 4 5 x U,
\]
we further deduce that $\Phi = \Phi' = \Phi'' = 0$ at $x = \pm \ell$ and $\Phi''' > 0$ at $x = \pm \ell$. This entails by symmetry $\Phi(x) = \Phi(-x)$ and because of $U > 0$ in $(-\ell,\ell)$ that $\Phi > 0$ in $(-\ell,\ell)$, and we infer by continuity that \eqref{bounds-phi}, \eqref{bounds-phip}, and \eqref{bounds-phipp} must hold.
\end{proof}

We use the following elementary estimates:
\begin{lemma}[see \cite{gnann2015well}]\label{lem-hardy-lx}
For $\beta > -1$ and $\gamma \in \R$ it holds
\begin{equation}\label{hardy-lx}
\int_{-\ell}^\ell (\ell^2-x^2)^\beta V^2 \, \d x \lesssim_{\beta,\gamma,\ell} \int_{-\ell}^\ell \big((\ell^2-x^2)^\gamma V^2 + (\ell^2-x^2)^{\beta+2} (\partial_x V)^2\big) \, \d x \quad \text{for all } V \in C^\infty([-\ell,\ell]).
\end{equation}
Furthermore, estimate~\eqref{hardy-lx} also holds for $\beta < -1$ provided $V = 0$ at $x = \pm \ell$.
\end{lemma}
\begin{proof}
For $\beta > -1$ estimate~\eqref{hardy-lx} follows by scaling from \cite[Lemma~3.2]{gnann2015well}, where Hardy's inequality near the boundaries $x = \pm \ell$ is applied. For $\beta < -1$ the same proof applies provided $V = 0$ at $x = \pm \ell$.
\end{proof}
\begin{lemma}[Interpolation inequality]\label{lem-interp}
For $k, m, n \ge 0$ such that $k \ge m$ and $0 < \eps \le 1$ it holds
\begin{equation}\label{interp-ineq}
[V]_k \lesssim_{k,\ell,m,n} \eps^{-1} [V]_{k-m} + \eps [V]_{k+n} \quad \text{for all } V \in \cH^{k+n}.
\end{equation}
\end{lemma}
\begin{proof}
For $k \ge 1$ and $V \in C^\infty([-\ell,\ell])$ we have
\begin{align*}
[V]_k^2 &= \int_{-\ell}^\ell (\ell^2-x^2)^{k+2} (\partial_x^k V)^2 \, \d x \\
&= (k+2) \int_{-\ell}^\ell x (\ell^2-x^2)^{k+1} \partial_x (\partial_x^{k-1} V)^2 \, \d x - \int_{-\ell}^\ell (\ell^2-x^2)^{k+2} (\partial_x^{k-1} V) (\partial_x^{k+1} V) \, \d x \\
&= (k+1) (k+2) \int_{-\ell}^\ell x^2 (\ell^2-x^2)^k (\partial_x^{k-1} V)^2 \, \d x - (k+2) \int_{-\ell}^\ell (\ell^2-x^2)^{k+1} (\partial_x^{k-1} V)^2 \, \d x \\
&\phantom{=} - \int_{-\ell}^\ell (\ell^2-x^2)^{k+2} (\partial_x^{k-1} V) (\partial_x^{k+1} V) \, \d x \\
&\le 2 (k+1) (k+2) \ell^2 [V]_{k-1} \Big(\int_{-\ell}^\ell (\ell^2-x^2)^{k-1} (\partial_x^{k-1} V)^2 \, \d x\Big)^{\frac 1 2} - (k+2) [V]_{k-1}^2 + [V]_{k-1} [V]_{k+1}.
\end{align*}
By \eqref{hardy-lx} of Lemma~\ref{lem-hardy-lx} it then follows
\[
\Big(\int_{-\ell}^\ell (\ell^2-x^2)^{k-1} (\partial_x^{k-1} V)^2 \, \d x\Big)^{\frac 1 2} \lesssim_{k,\ell} [V]_{k-1} + [V]_k + [V]_{k+1},
\]
so that after absorption of the term $\sim [V]_{k-1} [V]_k$ by Young's inequality, we end up with
\begin{equation*}
[V]_k^2 \lesssim_{k,\ell} [V]_{k-1}^2 + [V]_{k-1} [V]_{k+1}.
\end{equation*}
Now assume that
\begin{equation}\label{inter-step-ind}
[V]_k^2 \lesssim_{k,\ell,n_0} [V]_{k-m}^2 + [V]_{k-m} [V]_{k+n},
\end{equation}
for all $0 \le m \le \min\{k,n_0\}$, $0 \le n \le n_0$, where $n_0 \ge 1$. We can then use \eqref{inter-step-ind} with $k$ replaced by $k+n_0$, $m = n_0$, and $n = 1$ to estimate
\[
[V]_{k+n_0}^2 \lesssim_{k,\ell,n_0} [V]_k^2 + [V]_k [V]_{k+n_0+1}
\]
and if $k > n_0$, we again use \eqref{inter-step-ind} with $k$ replaced by $k-n_0$, $m = 1$, $n = n_0$, so that
\[
[V]_{k-n_0}^2 \lesssim_{k,\ell,n_0} [V]_{k-n_0-1}^2 + [V]_{k-n_0-1} [V]_k.
\]
In both cases, absorption of $[V]_k$ to the left-hand side entails that \eqref{inter-step-ind} is valid with $n_0$ replaced by $n_0+1$. This proves \eqref{interp-ineq} on applying Young's inequality.
\end{proof}
\begin{lemma}[Coercivity and symmetry]\label{lem-coerc}
We have
\begin{subequations}\label{coercivity-l-2}
\begin{align}
(V, \cL W)_U &= \frac 1 5 (V,W)_U + (\sqrt\Phi \partial_x V, \sqrt\Phi \partial_x W)_{L^2(-\ell,\ell)} + (U \partial_x^2 V,U \partial_x^2 W)_{L^2(-\ell,\ell)} \eqqcolon (V,W)_{U,2}, \quad V,W \in \cH^4, \label{coercivity-l-2-1}\\
|V|_{U,2} &\coloneqq \sqrt{(V,V)_{U,2}} \gtrsim |V|_2^2, \quad V \in \cH^2. \label{coercivity-l-2-2}
\end{align}
\end{subequations}
\end{lemma}
\begin{proof}
The first equality \eqref{coercivity-l-2-1} is immediate from the definition \eqref{formula-l} of $\cL$ through integration by parts, and was already anticipated in \eqref{coercivity-l} by considering the underlying gradient flow of the nonlinear problem. Estimate~\eqref{coercivity-l-2-2} then follows from \eqref{bounds-phi} and \eqref{bounds-u} of Lemma~\ref{lem-bounds-uphi}.
\end{proof}

Next, we first establish elliptic estimates on $\cL$:
\begin{lemma}[Elliptic regularity]\label{lem-elliptic}
We have for $k \ge 0$,
\begin{equation}\label{est-l}
|\cL V|_k \sim_{k,\ell} |V|_{k+4} \quad \text{for all } V \in \cH^{k+4}. 
\end{equation}
\end{lemma}
\begin{proof}
The bounds \eqref{est-u-ell^2-x^2} of Lemma~\ref{lem-bounds-uphi} and estimate~\eqref{hardy-lx} of Lemma~\ref{lem-hardy-lx} yield $|\cL V|_k \lesssim_{k,\ell} |V|_{k+4}$ in \eqref{est-l}. We first prove \eqref{est-l} for $k = 0$. Without loss of generality, we assume that $V \in C^\infty([-\ell,\ell])$. From \eqref{formula-l} it follows through integration by parts
\begin{align}
|\cL V|_U^2 &= \int_{-\ell}^\ell U^{-1} (\partial_x^2 (U^2 \partial_x^2 V))^2 \, \d x + \int_{-\ell}^\ell U^{-1} (\partial_x (\Phi \partial_x V))^2 \, \d x + \frac{1}{25} \int_{-\ell}^\ell U V^2 \, \d x \nonumber \\
& \phantom{=} - 2 \int_{-\ell}^\ell U^{-1} (\partial_x (\Phi \partial_x V)) (\partial_x^2 (U^2 \partial_x^2 V)) \, \d x + \frac 2 5 \int_{-\ell}^\ell V (\partial_x^2 (U^2 \partial_x^2 V)) \, \d x - \frac{2}{5} \int_{-\ell}^\ell V (\partial_x (\Phi \partial_x V)) \, \d x \nonumber \\
&= \int_{-\ell}^\ell U^{-1} (\partial_x^2 (U^2 \partial_x^2 V))^2 \, \d x - 2 \int_{-\ell}^\ell U^{-1} (\partial_x (\Phi \partial_x V)) (\partial_x^2 (U^2 \partial_x^2 V)) \, \d x + \int_{-\ell}^\ell U^{-1} (\partial_x (\Phi \partial_x V))^2 \, \d x \nonumber \\
& \phantom{=}  + \frac 2 5 \int_{-\ell}^\ell U^2 (\partial_x^2 V)^2 \, \d x + \frac{2}{5} \int_{-\ell}^\ell \Phi (\partial_x V)^2 \, \d x + \frac{1}{25} \int_{-\ell}^\ell U V^2 \, \d x. \nonumber
\end{align}
With help of Young's inequality it follows
\begin{align*}
|\cL V|_U^2 &\ge \eps \int_{-\ell}^\ell U^{-1} (\partial_x^2 (U^2 \partial_x^2 V))^2 \, \d x - \frac{\eps}{1-\eps} \int_{-\ell}^\ell U^{-1} (\partial_x (\Phi \partial_x V))^2 \, \d x \nonumber \\
& \phantom{=}  + \frac 2 5 \int_{-\ell}^\ell U^2 (\partial_x^2 V)^2 \, \d x + \frac{2}{5} \int_{-\ell}^\ell \Phi (\partial_x V)^2 \, \d x + \frac{1}{25} \int_{-\ell}^\ell U V^2 \, \d x
\end{align*}
for $0 < \eps < 1$. For $\eps \ll 1$ Lemma~\ref{lem-bounds-uphi} and Lemma~\ref{lem-hardy-lx} entail
\begin{align*}
\frac{\eps}{1-\eps} \int_{-\ell}^\ell U^{-1} (\partial_x (\Phi \partial_x V))^2 \, \d x \; &\lesssim \; \eps \int_{-\ell}^\ell U^{-1} \Phi^2 (\partial_x^2 V)^2 \, \d x + \eps \int_{-\ell}^\ell U^{-1} (\Phi')^2 (\partial_x V)^2 \, \d x \\
&\stackrel{\mathclap{\eqref{est-u-ell^2-x^2}}}{\lesssim_\ell} \; \eps \int_{-\ell}^\ell (\ell^2-x^2)^4 (\partial_x^2 V)^2 \, \d x + \eps \int_{-\ell}^\ell (\ell^2-x^2)^2 (\partial_x V)^2 \, \d x \\
&\stackrel{\mathclap{\eqref{hardy-lx}}}{\lesssim_\ell} \; \eps \int_{-\ell}^\ell (\ell^2-x^2)^4 (\partial_x^2 V)^2 \, \d x + \eps \int_{-\ell}^\ell (\ell^2-x^2)^3 (\partial_x V)^2 \, \d x \\
&\stackrel{\mathclap{\eqref{est-u-ell^2-x^2}}}{\lesssim_\ell} \; \eps \int_{-\ell}^\ell U^2 (\partial_x^2 V)^2 \, \d x + \eps \int_{-\ell}^\ell \Phi (\partial_x V)^3 \, \d x,
\end{align*}
at which point one can absorb for $\eps \ll 1$, so that
\begin{equation}\label{est-lv-intermediate}
|\cL V|_U^2 \gtrsim_\ell \int_{-\ell}^\ell U^{-1} (\partial_x^2 (U^2 \partial_x^2 V))^2 \, \d x + |V|_2^2.
\end{equation}
Since $U^2 \partial_x^2 V = \partial_x(U^2 \partial_x^2 V) = 0$ at $x = \pm\ell$, Lemma~\ref{lem-hardy-lx} is applicable, so that using Lemma~\ref{lem-bounds-uphi} we obtain
\begin{align*}
\int_{-\ell}^\ell U (\partial_x^2 V)^2 \, \d x \; &\stackrel{\mathclap{\eqref{bounds-u}}}{\sim_\ell} \; \int_{-\ell}^\ell (\ell^2-x^2)^{-6} (U^2 \partial_x^2 V)^2 \, \d x \\
&\stackrel{\mathclap{\eqref{hardy-lx}}}{\lesssim_\ell} \; \int_{-\ell}^\ell (\ell^2-x^2)^{-2} (U^2 \partial_x^2 V)^2 \, \d x + \int_{-\ell}^\ell (\ell^2-x^2)^{-4} (\partial_x(U^2 \partial_x^2 V))^2 \, \d x \\
&\stackrel{\mathclap{\eqref{est-lv-intermediate}}}{\lesssim_\ell} \; |\cL V|_U^2 + \int_{-\ell}^\ell U^{-2} (\partial_x (U^2 \partial_x^2 V))^2 \, \d x, \\
\int_{-\ell}^\ell U^{-2} (\partial_x (U^2 \partial_x^2 V))^2 \, \d x \; &\stackrel{\mathclap{\eqref{bounds-u}}}{\sim_\ell} \; \int_{-\ell}^\ell (\ell^2-x^2)^{-4} (\partial_x (U^2 \partial_x^2 V))^2 \, \d x \\
&\stackrel{\mathclap{\eqref{hardy-lx}}}{\lesssim_\ell} \; \int_{-\ell}^\ell (\ell^2-x^2)^{-2} (\partial_x (U^2 \partial_x^2 V))^2 \, \d x + \int_{-\ell}^\ell (\ell^2-x^2)^{-2} (\partial_x^2 (U^2 \partial_x^2 V))^2 \, \d x \\
&\stackrel{\mathclap{\eqref{est-lv-intermediate}}}{\lesssim_\ell} \; [V]_3^2 + |\cL V|_U^2.
\end{align*}
This implies with Lemma~\ref{lem-bounds-uphi} and using interpolation according to Lemma~\ref{lem-interp},
\begin{align*}
[V]_4^2 \; &\stackrel{\mathclap{\eqref{est-u-ell^2-x^2}}}{\lesssim_\ell} \; \int_{-\ell}^\ell U^{-1} (\partial_x^2 (U^2 \partial_x^2 V))^2 \, \d x + \int_{-\ell}^\ell U^{-2} (\partial_x (U^2 \partial_x^2 V))^2 \, \d x + \int_{-\ell}^\ell U (\partial_x^2 V)^2 \, \d x \\
&\stackrel{\mathclap{\eqref{est-lv-intermediate}}}{\lesssim_\ell} \;  [V]_3^2 + |\cL V|_U^2 \\
&\stackrel{\mathclap{\eqref{interp-ineq}}}{\lesssim_\ell} \; [V]_2^2 + [V]_2 [V]_4 + |\cL V|_U^2,
\end{align*}
so that after applying Young's inequality and using \eqref{est-lv-intermediate} once more,
\[
|V|_2 + [V]_4 \lesssim_\ell |\cL V|_U.
\]
Utilizing the interpolation inequality \eqref{interp-ineq} of Lemma~\ref{lem-interp} to obtain control on $[V]_3$ then concludes the proof of \eqref{est-l} for $k = 0$.

\medskip

Now assume that \eqref{est-l} is valid for $k \in \N_0$. We use Lemma~\ref{lem-der-l}, Lemma~\ref{lem-bounds-uphi}, and Lemma~\ref{lem-hardy-lx}, to deduce for the principal part of $\partial_x^{k+1} \cL$
\begin{align}\nonumber
\int_{-\ell}^\ell (\ell^2-x^2)^{-k-3} \big(\partial_x^2 (U^{\frac{k+5}{2}} \partial_x^{k+3} V)\big)^2 \, \d x \quad &\stackrel{\mathclap{\eqref{est-u-ell^2-x^2}}}{\lesssim_\ell} \quad \int_{-\ell}^\ell (\ell^2-x^2)^{k+3} \big(U^{-\frac{k+3}{2}} \partial_x^2 (U^{\frac{k+5}{2}} \partial_x^{k+3} V)\big)^2 \, \d x \\
&\stackrel{\mathclap{\eqref{derivatives-l-2},\eqref{hardy-lx}}}{\lesssim_{k,\ell}} \quad [\cL V]_{k+1}^2 + \sum_{j = \min\{k+1,2\}}^{k+3} [V]_j^2 \nonumber \\
&\stackrel{\mathclap{\eqref{est-l}}}{\lesssim_\ell} \quad |\cL V|_{k+1}^2. \label{higher-intermediate}
\end{align}
Now we use Lemma~\ref{lem-bounds-uphi} and Lemma~\ref{lem-hardy-lx}, resulting in
\begin{align*}
\int_{-\ell}^\ell (\ell^2-x^2)^{-k-7} \big(U^{\frac{k+5}{2}} \partial_x^{k+3} V\big)^2 \, \d x \quad &\stackrel{\mathclap{\eqref{est-u-ell^2-x^2},\eqref{hardy-lx}}}{\lesssim_\ell} \quad [V]_{k+3}^2 + \int_{-\ell}^\ell (\ell^2-x^2)^{-k-5} \big(\partial_x(U^{\frac{k+5}{2}} \partial_x^{k+3} V)\big)^2 \, \d x \\
&\stackrel{\mathclap{\eqref{est-l}}}{\lesssim_\ell} \quad |\cL V|_k^2 + \int_{-\ell}^\ell (\ell^2-x^2)^{-k-5} \big(\partial_x(U^{\frac{k+5}{2}} \partial_x^{k+3} V)\big)^2 \, \d x, \\
\int_{-\ell}^\ell (\ell^2-x^2)^{-k-5} \big(\partial_x(U^{\frac{k+5}{2}} \partial_x^{k+3} V)\big)^2 \, \d x \quad &\stackrel{\mathclap{\eqref{est-u-ell^2-x^2},\eqref{hardy-lx}}}{\lesssim_\ell} \quad [V]_{k+4}^2 + \int_{-\ell}^\ell (\ell^2-x^2)^{-k-3} \big(\partial_x^2 (U^{\frac{k+5}{2}} \partial_x^{k+3} V)\big)^2 \, \d x \\
&\stackrel{\mathclap{\eqref{higher-intermediate}}}{\lesssim_\ell} \quad |\cL V|_{k+1}^2.
\end{align*}
This implies with emma~\ref{lem-bounds-uphi}
\begin{align*}
[V]_{k+5}^2 \; &\stackrel{\mathclap{\eqref{est-u-ell^2-x^2}}}{\lesssim_\ell} \; \int_{-\ell}^\ell (\ell^2-x^2)^{-k-3} \big(\partial_x^2 (U^{\frac{k+5}{2}} \partial_x^{k+3} V)\big)^2 \, \d x + \int_{-\ell}^\ell (\ell^2-x^2)^{-k-5} \big(\partial_x(U^{\frac{k+5}{2}} \partial_x^{k+3} V)\big)^2 \, \d x \\
&\phantom{\lesssim_\ell} \; + \int_{-\ell}^\ell (\ell^2-x^2)^{-k-7} \big(U^{\frac{k+5}{2}} \partial_x^{k+3} V\big)^2 \, \d x \\
&\stackrel{\mathclap{\eqref{higher-intermediate}}}{\lesssim_{k,\ell}} \; |\cL V|_{k+1}^2.
\end{align*}
Hence, the elliptic estimate \eqref{est-l} follows by complete induction over $k \in \N_0$.
\end{proof}
\subsection{Resolvent estimates}\label{sec-resolvent}
Now, consider the corresponding resolvent problem
\begin{equation}\label{resolvent}
\lambda V + \cL V = \lambda V^{(0)} + F \quad \text{in } (-\ell,\ell),
\end{equation}
for given functions $V^{(0)}, F \colon (-\ell,\ell) \to \R$, where $\frac 1 \lambda > 0$ is a fixed time step.

\begin{lemma}[Resolvent estimate]\label{lem-resolvent}
For $k \in \N_0$, $0 < \lambda < \infty$, $F \in \cH^{2k}$, and $V^{(0)} \in \cH^{2k+2}$, we have a unique solution $u \in \cH^{2k+4}$ of \eqref{resolvent} satisfying the resolvent estimate
\begin{subequations}\label{resolvent-est}
\begin{align}
k \text{ is even} &\colon \quad  \lambda |\cL^{\frac k 2} V|_{U,2}^2 + |\cL^{\frac k 2 + 1} V|_U^2 \le \lambda |\cL^{\frac k 2} V^{(0)}|_{U,2}^2 + |\cL^{\frac k 2} F|_U^2, \\
k \text{ is odd} &\colon \quad \lambda |\cL^{\frac{k+1}{2}} V|_U^2 + |\cL^{\frac{k+1}{2}} V|_{U,2}^2 \le \lambda |\cL^{\frac{k+1}{2}} V^{(0)}|_U^2 + |\cL^{\frac{k-1}{2}} F|_{U,2}^2,
\end{align} 
\end{subequations}
where $|\cdot|_{U,2}$ was introduced in \eqref{coercivity-l-2-2}.
\end{lemma}
\begin{proof}
We define the bilinear forms
\begin{align*}
B_{\lambda,k} \colon & \cH^{2k+4} \times \cH^{2k+4} \to \R, \\
& (V,W) \mapsto B_{\lambda,k}[V,W] \coloneqq \begin{cases} \lambda (\cL^{\frac k 2} V,\cL^{\frac k 2} W)_{U,2} + (\cL^{\frac {k+2} 2} V,\cL^{\frac {k+2} 2} W)_U, & k \text{ is even}, \\
\lambda (\cL^{\frac{k+1}{2}} V,\cL^{\frac{k+1}{2}} W)_U + (\cL^{\frac{k+1}{2}} V
,\cL^{\frac{k+1}{2}} W)_{U,2}, & k \text{ is odd},
\end{cases}
\end{align*}
where
\[
(V,W)_{U,2} \stackrel{\eqref{coercivity-l-2-1}}{=} \frac 1 5 (V,W)_U + (\sqrt\Phi \partial_x V,\sqrt\Phi \partial_x W)_{L^2(-\ell,\ell)} + (U \partial_x^2 V,U \partial_x^2 V)_{L^2(-\ell,\ell)}.
\]
From the expression \eqref{derivatives-l} of Lemma~\ref{lem-der-l} of (derivatives of) $\cL$ and using Hardy's inequality close to the boundaries $x = \pm \ell$ (see \eqref{hardy-lx} of Lemma~\ref{lem-hardy-lx}), we infer that $B_{\lambda,k}$ is well-defined and continuous. Furthermore, for $V \in \cH^{2k+4}$ it holds, using \eqref{est-u-ell^2-x^2} of Lemma~\ref{lem-bounds-uphi} and \eqref{est-l} of Lemma~\ref{lem-elliptic},
\begin{align*}
k \text{ is even} \colon& \quad B_{\lambda,k}[V,V] \ge |\cL^{\frac{k+2}{2}} V|_U^2 \gtrsim_\ell |\cL^{\frac k 2} V|_4^2 \gtrsim \ldots \gtrsim_{k,\ell} |V|_{2k+4}, \\
k \text{ is odd} \colon& \quad B_{\lambda,k}[V,V] \ge |\cL^{\frac{k+1}{2}}V|_{U,2}^2 \gtrsim_\ell |\cL^{\frac{k+1}{2}}V|_2^2 \gtrsim_\ell \ldots \gtrsim_{k,\ell} |V|_{2k+4}^2,
\end{align*}
establishing coercivity of $B_k$.

\medskip

Now define
\begin{align*}
G_{\lambda,k} \colon \cH^{2k+4} \to \R, \quad W \mapsto G_{\lambda,k}[W] \coloneqq \begin{cases} \lambda (\cL^{\frac k 2} V^{(0)}, \cL^{\frac k 2} W)_{U,2} + (\cL^{\frac k 2} F, \cL^{\frac {k+2} 2} W)_U, & k \text{ is even}, \\
\lambda (\cL^{\frac{k+1} 2} V^{(0)}, \cL^{\frac{k+1} 2} W)_U + (\cL^{\frac{k-1} 2} F, \cL^{\frac{k+1} 2} W)_{U,2}, & k \text{ is odd}. \end{cases}
\end{align*}
From \eqref{derivatives-l} of Lemma~\ref{lem-der-l}, \eqref{est-u-ell^2-x^2} of Lemma~\ref{lem-bounds-uphi}, and \eqref{hardy-lx} of Lemma~\ref{lem-hardy-lx} one recognizes that $G_{\lambda,k}$ is continuous. Hence, by the Lax-Milgram theorem, there exists a unique solution $V \in \cH^{2k+4}$ to
\begin{equation}\label{bilinear-eq}
B_{\lambda,k}[V,W] = G_{\lambda,k}[W] \quad \text{for all } W \in \cH^{2k+4}.
\end{equation}
By symmetry \eqref{coercivity-l-2-1} stated in Lemma~\ref{lem-coerc}, we infer
\begin{align*}
\big(\lambda V + \cL V - \lambda V^{(0)} - F, \cL^{k+1} W\big)_U = 0 \quad \text{for all } W \in C^\infty([-\ell,\ell]).
\end{align*}
In order to conclude that \eqref{resolvent} is satisfied, it suffices to solve $\cL^{k+1} W = \tilde W$ for $\tilde W \in L^2_U$ and $W \in \cH^{4k+4}$. The latter can be achieved as before by iteratively using the Lax-Milgram theorem.

\medskip

As a last step, we evaluate \eqref{bilinear-eq} at $W = V$ and obtain
\begin{align*}
k \text{ is even} &\colon \quad  \lambda |\cL^{\frac k 2} V|_{U,2}^2 + |\cL^{\frac{k+2}{2}} V|_U^2 = \lambda (\cL^{\frac k 2} V^{(0)},\cL^{\frac k 2} V)_{U,2} + (\cL^{\frac k 2} F,\cL^{\frac {k+2} 2} V)_U, \\
k \text{ is odd} &\colon \quad  \lambda |\cL^{\frac{k+1}{2}} V|_U^2 + |\cL^{\frac{k+1}{2}} V|_{U,2}^2 = \lambda (\cL^{\frac{k+1}{2}} V^{(0)},\cL^{\frac{k+1}{2}} V)_U + (\cL^{\frac{k-1}{2}} F, \cL^{\frac{k+1}{2}} V)_{U,2},
\end{align*} 
so that with Young's inequality we arrive at \eqref{resolvent-est}.
\end{proof}
\subsection{Parabolic estimates}\label{sec-parabolic}
Next, we turn to the time-discretized version of \eqref{lin-pde} and are able to prove:
\begin{lemma}[Time-discretized estimates]\label{lem-time-discrete}
Given $k \in \N_0$, $0 < h \le 1$, a sequence $(F_j)_{j \in \N_0} \in (\cH^{4k})^{\N_0}$, and $V^{(0)} \in \cH^{2k+2}$, there exists a sequence $(V_j)_{j \in \N_0} \in (\cH^{4k+2})^{\N_0}$ solving
\begin{equation}\label{lin-pde-discrete}
\frac 1 h (V_j - V_{j-1}) + \cL V_j = F_{j-1} \quad \text{for } j \in \N,
\end{equation}
and satisfying the following {\it \`a-priori} estimates for $j \ge 0$,
\begin{align}
(j h)^{k-2} |V_j|_{2k+2}^2 + \sum_{j' = 1}^j h (j' h)^{k-2} |V_{j'}|_{2k+4}^2 &\lesssim_{k,\ell} \sum_{j' = 0}^{j-1} h (j h)^{k-2} |F_{j'}|_{2k}^2 \nonumber \\
&\phantom{\lesssim_{k,\ell}} + (k-2) \sum_{j' = 1}^j h ((j'+1) h)^{k-3} |V_{j'}|_{2k+2}^2 && \text{for } k \ge 3. \label{a-priori-discrete}
\end{align}
\end{lemma}
\begin{proof}
Notice that \eqref{lin-pde-discrete} is equivalent to \eqref{resolvent} on setting $\lambda \coloneqq h^{-1}$. This ensures existence of the sequence $(V_j)_{j \in \N_0}$ by iteratively applying Lemma~\ref{lem-resolvent}.

\medskip

In order to prove \eqref{a-priori-discrete} 
for $k \ge 3$, we first multiply estimates~\eqref{resolvent-est} with $(jh)^{k-2}$ before summing and obtain
\begin{align*}
(jh)^{k-2} |\cL^{\frac k 2} V_j|_{U,2}^2 + \sum_{j' = 1}^j h (j'h)^{k-2} |\cL^{\frac k 2 + 1} V|_U^2 &\le \sum_{j' = 0}^{j-1} h (j'h)^{k-2} |\cL^{\frac k 2} F_j|_U^2 \\
&\phantom{\le} + \sum_{j' = 0}^{j-1} (((j'+1)h)^{k-2}-(j'h)^{k-2}) |\cL^{\frac k 2} F_{j'}|_U^2
\end{align*}
if $k$ is even and
\begin{align*}
(jh)^{k-2} |\cL^{\frac{k+1}{2}} V_j|_U^2 + \sum_{j' = 1}^j h (j'h)^{k-2} |\cL^{\frac{k+1}{2}} V_j|_{U,2}^2 &\le \sum_{j' = 0}^{j-1} h (j'h)^{k-2}  |\cL^{\frac{k-1}{2}} F_j|_{U,2}^2 \\
&\phantom{\le} + \sum_{j' = 0}^{j-1} (((j'+1)h)^{k-2}-(j'h)^{k-2}) |\cL^{\frac {k-1} 2} F_{j'}|_{U,2}^2.
\end{align*} 
if $k$ is odd. We then notice that
\[
((j'+1)h)^{k-2}-(j'h)^{k-2} \le (k-2) ((j'+1)h)^{k-3} h,
\]
which entails with help of estimates~\eqref{est-u-ell^2-x^2} of Lemma~\ref{lem-bounds-uphi}, estimate~\eqref{coercivity-l-2-2} of Lemma~\ref{lem-coerc}, and estimate~\eqref{est-l} of Lemma~\ref{lem-elliptic}, that \eqref{a-priori-discrete} 
is valid.
\end{proof}
\begin{proposition}[Maximal regularity of the linear evolution]\label{prop-maxreg}
Suppose $k \in \N$ with $k \ge 2$,
\[
F \in L^2_\mathrm{loc}((0,\infty);\cH^{2k}) \quad \text{such that} \quad (t \mapsto t^{\frac{k'-2}{2}} F(t)) \in L^2(0,\infty;\cH^{2k'}) \quad \text{for} \quad k' \in \{2,\ldots,k\},
\]
and $V^{(0)} \in \cH^6$. Then there exist a unique solution
\[
V \in H^1_\mathrm{loc}((0,\infty);\cH^{2k}) \cap L^2_\mathrm{loc}((0,\infty);\cH^{2k+4}) \cap C^0((0,\infty); \cH^{2k+2}),
\]
such that
\begin{subequations}\label{regularity-weighted}
\begin{align}
(s \mapsto s^{\frac{k'-2}{2}} \partial_s^m V(s)) &\in L^2(0,\infty;\cH^{2k'+4-4m}) && \text{for} \quad k' \in \{2,\ldots,k\}, \quad m \in \{0,1\}, \\
(s \mapsto s^{\frac{k'-2}{2}} V(s)) &\in C^0([0,\infty);\cH^{2k'+2}) && \text{for} \quad k' \in \{2,\ldots,k\},
\end{align}
\end{subequations}
to \eqref{lin-pde}. This solution satisfies the {\it \`a-priori} estimate
\begin{align}\nonumber
& \sum_{k' = 2}^k \Big(s^{k'-2} |V(s)|_{2k'+2}^2 + \int_0^s (s')^{k'-2} \big(|(\partial_s V)(s')|_{2k'}^2 + |V(s')|_{2k'+4}^2\big) \, \d s'\Big) \\
& \quad \lesssim_{k,\ell} |V^{(0)}|_6^2 + \sum_{k' = 2}^k \int_0^s (s')^{k'-2} |F(s')|_{2k'}^2 \, \d s'. \label{a-priori-lin}
\end{align}
\end{proposition}
\begin{proof}
For $0 < h \le 1$, assume by approximation (continuous extension of the linear solution operator) $F \in C^\infty_\mathrm{c}((0,\infty);\cH^{2k})$, and define $F_j \coloneqq \frac 1 h \int_{(j-1)h}^{jh} F(s) \, \d s$ for $j \in \N_0$ and $V_j$ for $j \in \N$ by applying Lemma~\ref{lem-time-discrete}. We can then define discrete approximations
\begin{align*}
V_h(s) \coloneqq V_j, \quad F_h(s) \coloneqq F_j, \quad \text{and} \quad \tilde V_h(s) \coloneqq \frac{s-(j-1)h}{h} V_j + \frac{jh-s}{h} V_{j-1} \quad \text{for} \quad (j-1) h < s \le j h.
\end{align*}
By \eqref{lin-pde-discrete} we then have almost everywhere the PDE
\begin{equation}\label{lin-pde-approx-h}
\partial_s \tilde V_h + \cL V_h = F_h.
\end{equation}
Summing estimates \eqref{a-priori-discrete} 
in conjunction with \eqref{lin-pde-approx-h} entails
\begin{align}\nonumber
& \sum_{k' = 2}^k \Big(s^{k'-2} |V_h(s)|_{2k'+2}^2 + \int_0^s (s')^{k'-2} \big(|(\partial_s \tilde V_h)(s')|_{2k'}^2 + |V_h(s')|_{2k'+4}^2\big) \, \d s'\Big) \\
& \quad \lesssim_{k,\ell} |V^{(0)}|_6^2 + \sum_{k' = 2}^k \int_0^s (s')^{k'-2} |F_h(s')|_{2k'}^2 \, \d s', \label{a-priori-lin-h}
\end{align}
which corresponds to \eqref{a-priori-lin} on the time-discretized level.

\medskip

In order to establish convergence, observe that for $j \in \N_0$ such that $s+h \ge j h \ge s$, and assuming $0 < h \le 1$, it holds
\begin{align}\nonumber
\int_0^s (s')^{k'-2} |F_h(s')-F(s')|_{2k'}^2 \, \d s' &\le \sum_{j' = 0}^{j-1} \int_{j' h}^{(j'+1) h} (s')^{k'-2} \Big|\frac 1 h \int_{j' h}^{(j'+1)h} F(s'') \, \d s'' - F(s')\Big|_{2k'}^2 \, \d s' \\
&\le \sum_{j' = 0}^{j-1} \frac 1 h \int_{j' h}^{(j'+1)h} \int_{j' h}^{(j'+1)h} (s')^{k'-2} |F(s'') - F(s')|_{2k'}^2 \, \d s'' \, \d s' \nonumber\\
&\le (j h)^{k'-2} \sup_{0 \le s' \le jh} |(\partial_s F)(s')|_{2k'}^2 j h^3 \nonumber\\
&\le \Big((s+1)^{k'-1} \sup_{0 \le s' \le s+1} |(\partial_s F)(s')|_{2k'}^2\Big) h^2 \to 0 \quad \text{as} \quad h \searrow 0. \label{convergence-f}
\end{align}
Hence, weak-$*$ sequential compactness in \eqref{a-priori-lin-h} entails that subsequences of $(V_h)_{h > 0}$ and $(\tilde V_h)_{h > 0}$, again denoted as such, weak-$*$ converge in the left-hand side norms of \eqref{a-priori-lin-h} to locally integrable functions $V, \tilde V \colon (0,\infty) \times (-\ell,\ell) \to \R$, and \eqref{a-priori-lin} is valid by weak lower-semicontinuity of the appearing norms. For $\phi \in C^\infty_\mathrm{c}((0,\infty) \times (-\ell,\ell))$ it holds for a $j \in \N$ sufficiently large,
\begin{align*}
\int_0^\infty \int_{-\ell}^\ell U \phi (V_h-\tilde V_h) \, \d x \, \d s \; &\le \; \sum_{j' = 1}^j \int_{(j'-1) h}^{j'h} \int_{-\ell}^\ell U |\phi| |V_{j'}-V_{j'+1}| \, \d x \, \d s \\
&\stackrel{\mathclap{\eqref{lin-pde-discrete}}}{\le} \; 
h^2 \sum_{j' = 1}^j (|\phi|_U^2 + |\cL V_{j'}|_U^2 + |F_{j'-1}|_U^2) \\
&\le \; h s |\phi|_U^2 + h \int_0^s (|V_h|_4^2 + |F_h|_0^2) \, \d s',
\end{align*}
where we have used \eqref{est-u-ell^2-x^2} of Lemma~\ref{lem-bounds-uphi} and \eqref{est-l} of Lemma~\ref{lem-elliptic} in the last inequality, and have chosen $s > 0$ sufficiently large but independent of $h$. With help of \eqref{a-priori-lin-h} and \eqref{convergence-f} we infer that
\[
\int_0^\infty \int_{-\ell}^\ell U \phi (V_h-\tilde V_h) \, \d x \, \d s \to 0 \quad \text{as} \quad h \searrow 0,
\]
so that we must have $V = \tilde V$. Then passing to the limit $h \searrow 0$ in the weak formulation of \eqref{lin-pde-approx-h} (in a similar way as carried out for identifying $V = \tilde V$), we conclude that \eqref{lin-pde} is satisfied.

\medskip

In order to prove continuity, observe that for $s \ge 0$ and $h \coloneqq \frac s N$ with $N \in \N$ it holds for $W \in C^1([0,\infty);\cH^{2k'+4})$,
\begin{align*}
&s^{k'-2} |W(s)|_{2k'+2}^2 - \delta_{k',2} |W(0)|_{2k'+2}^2 \\
& \quad = \sum_{j = 1}^N \big( (jh)^{k'-2} |W(jh)|_{2k'+2}^2 - ((j-1)h)^{k'-2} |W((j-1)h)|_{2k'+2}^2\big) \\
& \quad = \sum_{j = 1}^N \big((jh)^{k'-2} - ((j-1)h)^{k'-2}\big) |W(jh)|_{2k'+2}^2 + \sum_{j = 1}^N ((j-1)h)^{k'-2} \big(|W(jh)|_{2k'+2}^2 - |W((j-1)h)|_{2k'+2}^2\big) \\
& \quad \le (k'-2) \sum_{j = 1}^N h (j h)^{k'-3} |W(jh)|_{2k'+2}^2 \\
& \quad \phantom{\le} + \sum_{j = 1}^N ((j-1) h)^{k'-2} (W(jh) - W((j-1)h), W(jh)+W((j-1)h))_{2k'+2} \\
& \quad \stackrel{\eqref{interp-ineq}}{\lesssim_{k',\ell}} (k'-2) \sum_{j = 1}^N h (j h)^{k'-3} |W(jh)|_{2k'+2}^2 \\
& \quad \phantom{\lesssim_{k',\ell}} + \sum_{j = 1}^N h ((j-1) h)^{k'-2} \Big(\Big|\frac 1 h \int_{(j-1)h}^{jh} (\partial_s W)(s') \, \d s'\Big|_{2k'}^2 + |W(jh)|_{2k'+4}^2 + |W((j-1)h)|_{2k'+4}^2 \Big)
\end{align*}
where we have used Lemma~\ref{lem-interp} in the last estimate. Passage to the limit $h \searrow 0$ entails the trace estimate
\begin{align}\nonumber
\sup_{0 \le s' \le s} (s')^{k'-2} |W(s')|_{2k'+2}^2 &\lesssim_{k',\ell} \delta_{k',2} |W(0)|_{2k'+2}^2 + (k'-2) \int_0^s (s')^{k'-3} |W(s')|_{2k'+2}^2 \, \d s' \\
&\phantom{\lesssim_{k',\ell}} + \int_0^s (s')^{k'-2} \big(|(\partial_s W)(s')|_{2k'}^2 + |W(s')|_{2k'+4}^2\big) \, \d s'. \label{trace-est}
\end{align}
By approximation, \eqref{trace-est} is also valid for $W = V$ and hence the continuity statements follow (the approximation property with continuous functions applies to the right-hand side of \eqref{trace-est}, hence also to its left-hand side).

\medskip

It remains to prove uniqueness. Therefore, note that for $V \in H^1(0,\infty;L^2_U) \cap L^2(0,\infty;\cH^4) \cap C^0([0,\infty);\cH^2)$ solving
\[
\partial_s V + \cL V = 0, \quad \text{for} \quad s > 0, \qquad \text{subject to} \qquad V = 0 \quad \text{at} \quad s = 0,
\]
we obtain by testing and with help of \eqref{coercivity-l-2} of Lemma~\ref{lem-coerc} for some $c > 0$,
\[
\frac 1 2 \frac{\d}{\d s} |V|_U^2 + c |V|_2^2 \le 0 \quad \text{for} \quad s > 0,
\]
which after integration in $s$ and using $V = 0$ at $s = 0$ leads to $|V|_U = 0$ for all $s > 0$ and thus $V = 0$ almost everywhere.
\end{proof}
\section{The nonlinear degenerate-parabolic equation}\label{sec-nonlinear}
\subsection{Estimates on the nonlinearity}\label{sec-nonlinearity}
A main ingredient of all nonlinear estimates is the following consequence of Hardy's inequality
\begin{lemma}\label{lem-c0}
We have for $n \in \N_0$, $d \in \N_0$, and $k = 2d + 4-2n$,
\begin{equation}\label{c0_embed}
\|(\ell^2-x^2)^n \partial_x^d V\|_{C^0([-\ell,\ell])} \lesssim_{d,n} |V|_k
\end{equation}
\end{lemma}
More refined estimates using interpolation spaces (cf.~\cite{bringmann2016corrigendum,giacomelli2008smooth,gnann2015well}) allow to prove an analogous estimate with control on one derivative less (one can prove the embedding $(\cH^4,\cH^6)_{\frac 1 2,1} \hookrightarrow C^1([-\ell,\ell])$, where $(\cdot,\cdot)_{\theta,q}$ is the real-interpolation functor \cite{bennett1988interpolation}, see \cite[Lemma~3.5]{gnann2015well} for an analogous case). We thus obtain control on $\|V_x\|_{C^0([-\ell,\ell])}$ in estimate~\eqref{maxreg-general} for $k \ge 6$, leading to well-posedness and stability in this setting.

\begin{proof}[Proof of Lemma~\ref{lem-c0}]
Observe that by Sobolev embedding and afterwards applying estimate~\eqref{hardy-lx} of Lemma~\ref{lem-hardy-lx},
\begin{align*}
\|(\ell^2-x^2)^n \partial_x^d V\|_{C^0([-\ell,\ell])}^2 &\lesssim_{d,\ell\phantom{,n}} \int_{-\ell}^\ell (\ell^2-x^2)^{2n} \big((1+n^2 (\ell^2-x^2)^{-2}) (\partial_x^d V)^2 + (\partial_x^{d+1} V)^2\big) \, \d x \\
&\lesssim_{d,\ell,n} \sum_{j = d}^k \int_{-\ell}^\ell (\ell^2-x^2)^{j+2} (\partial_x^j V)^2 \, \d x \le |V|_k^2. \qedhere
\end{align*}
\end{proof}
\begin{proposition}[Estimate of the nonlinearity]\label{prop-nonlinear}
For $k \in \N_0$, there is a sufficiently large constant $C = C(k,\ell) < \infty$ such that, the estimate
\begin{equation}\label{est-nonlinearity}
|N[V]-N[W]|_k \le C \frac{(|V|_{k'+2}+|W|_{k'+2}) |V-W|_{k+4} + (|V|_{k+4}+|W|_{k+4}) |V-W|_{k'+2}}{1-C(|V|_{k'+2}+|W|_{k'+2})}
\end{equation}
holds for $V, W \in \cH^{k+4} \cap \cH^6$, provided $C (|V|_{k'+2}+|W|_{k'+2}) < 1$, where $k' \coloneqq \max\{\lfloor\frac{k+3}{2}\rfloor,4\}$.
\end{proposition}
\begin{proof}
From \eqref{grad-u-ez} and \eqref{nonlinearity} we infer
\[
N[V] = \cL V + \partial_Z^3 \Theta - \Theta^2 \partial_Z \Theta - \frac Z 5, \quad Z = x+V, \quad \partial_Z = (1+V_x)^{-1} \partial_x, \quad \Theta = U (1+V_x)^{-1}.
\]
On assuming $\|V_x\|_{C^0([0,\infty) \times [-\ell,\ell])} < 1$, we can expand
\[
(1+V_x)^{-1} = \sum_{m = 0}^\infty (-V_x)^m.
\]
It then holds (we use the convention that derivatives act on everything to their right and do not write parentheses for the sake of readability)
\begin{align*}
\partial_Z^3 \Theta &= \sum_{j = 1}^4 \sum_{m_j = 0}^\infty (-V_x)^{m_1} \partial_x (-V_x)^{m_2} \partial_x (-V_x)^{m_3} \partial_x (-V_x)^{m_4} U, \\
- \Theta^2 \partial_Z \Theta &= - U^2 \sum_{j = 1}^3 \sum_{m_j = 0}^\infty (-V_x)^{m_1} (-V_x)^{m_2} \partial_x (-V_x)^{m_3} U,
\end{align*}
so that
\begin{align}\label{structure-nv}
N[V] = \sum_{m \coloneqq \sum_{j = 1}^4 m_j \ge 2} (-1)^m V_x^{m_1} \partial_x V_x^{m_2} \partial_x V_x^{m_3} \partial_x V_x^{m_4} U  + U^2 \sum_{m \coloneqq \sum_{j = 1}^3 m_j \ge 2} (-1)^{m+1} V_x^{m_1+m_2} \partial_x V_x^{m_3} U.
\end{align}
We consider
\begin{align*}
|N[V]-N[W]|_k^2 = \sum_{r = 0}^k \int_{-\ell}^\ell (\ell^2-x^2)^{r+2} (\partial_x^r (N[V]-N[W]))^2 \, \d x,
\end{align*}
where due to \eqref{est-u-ell^2-x^2} of Lemma~\ref{lem-bounds-uphi}, we need to estimate the following terms (with polynomial pre-factors in $m_j$ with degree $\le k+3$ or $\le k+1$, respectively, that do not matter in a geometric series):
\begin{align*}
\text{terms I: } & \int_{-\ell}^\ell (\ell^2-x^2)^{r+2+2n} \prod_{\kappa = 1}^{m - q} V_{\kappa,x}^2 \prod_{p = 1}^q (\partial_x^{d_p} V_{m - q +p})^2 \, \d x, \\
& \quad \text{where} \quad V_\kappa \in \{V, W, V-W\}, \quad m \coloneqq \sum_{j = 1}^4 m_j \ge 2, \quad 0 \le r \le k, \quad 1 \le q \le \min\{m, r+3\}, \\
&  \quad \phantom{\text{where}} \quad d_{p} \ge d_{p+1} \ge 2, \quad 0 \le n \le 2, \quad \sum_{p = 1}^q (d_p-1) \le r+1+n,
\end{align*}
and
\begin{align*}
\text{terms II: } & \int_{-\ell}^\ell (\ell^2-x^2)^{r+2+2n} \prod_{\kappa = 1}^{m - q} V_{\kappa,x}^2 \prod_{p = 1}^q (\partial_x^{d_p} V_{m-q+p})^2 \, \d x, \\
& \quad \text{where} \quad V_\kappa \in \{V, W, V-W\}, \quad m \coloneqq \sum_{j = 1}^3 m_j \ge 2, \quad 0 \le r \le k, \quad 1 \le q \le \min\{m, r+1\}, \\
& \quad \phantom{\text{where}} \quad d_p \ge d_{p+1} \ge 2, \quad 5-r \le n \le 6, \quad \sum_{p = 1}^q (d_p-1) \le r+n-5.
\end{align*}
We can now estimate terms I and II,
\begin{align*}
& \int_{-\ell}^\ell (\ell^2-x^2)^{r+2+2n} \prod_{\kappa = 1}^{m - q} V_{\kappa,x}^2 \prod_{p = 1}^q (\partial_x^{d_p} V_{m - q +p})^2 \, \d x \\
& \quad \le \prod_{\kappa = 1}^{m - q} \|V_{\kappa,x}\|_{C^0([-\ell,\ell])}^2 \int_{-\ell}^\ell (\ell^2-x^2)^{r+2+2n} \prod_{p = 1}^q (\partial_x^{d_p} V_{m - q +p})^2 \, \d x \\
& \quad \lesssim_{k,\ell} \prod_{\kappa = 1}^{m - q} \|V_{\kappa,x}\|_{C^0([-\ell,\ell])}^2 \prod_{p = 2}^q \|(\ell^2-x^2)^{d_p-1} \partial_x^{d_p} V_{m-q+p}\|_{C^0([-\ell,\ell])}^2 \int_{-\ell}^\ell (\ell^2-x^2)^{2(d_1-1)-k} (\partial_x^{d_1} V_m)^2 \, \d x \\
& \quad \le \prod_{\kappa = 1}^{m - 1} (C |V_\kappa|_{k'+2})^2 (C |V_m|_{k+4})^2,
\end{align*}
where \eqref{hardy-lx} of Lemma~\ref{lem-hardy-lx} and \eqref{c0_embed} of Lemma~\ref{lem-c0} were used in the last line. Note that in case I we have used
\[
\sum_{p = 1}^q 2 (d_p-1) -k \le r + 2 + 2n \quad \Leftarrow \quad 2(r+1+n) - k \le r + 2 + 2n \quad \Leftarrow \quad r \le k,
\]
and
\[
d_p - 1 \le \frac{r+1+n}{2} \le \frac{k+3}{2} \quad \text{for} \quad p \ge 2, \quad \text{which implies} \quad d_p \le \frac{k+5}{2} \quad \text{for} \quad p \ge 2.
\]
In case II we have used
\[
\sum_{p = 1}^q 2 (d_p - 1) - k \le r + 2 + 2n \quad \Leftarrow \quad 2(r+n-5) - k \le r + 2 + 2n \quad \Leftarrow \quad r \le k,
\]
and
\[
d_p-1 \le \frac{r+n-5}{2} \le \frac{k+1}{2} \quad \text{for} \quad p \ge 2, \quad \text{which implies} \quad d_p \le \frac{k+3}{2} \quad \text{for} \quad p \ge 2.
\]
We thus obtain \eqref{est-nonlinearity} upon enlarging $C < \infty$.
\end{proof}
\subsection{Well-posedness and regularity of the nonlinear Cauchy problem}\label{sec-well}
The combination of maximal regularity of the linear equation \eqref{pde-linear} (cf.~Proposition~\ref{prop-maxreg}) and a suitable estimate of the nonlinearity \eqref{nonlinearity} (cf.~Proposition~\ref{prop-nonlinear}) leads to a well-posedness result for the nonlinear problem:

\begin{proof}[Proof of Theorem~\ref{th-well}]
We define the triple-bar norm $\tbar \cdot \tbar_k$ through
\[
\tbar V \tbar_k^2 \coloneqq \sum_{k' = 2}^k \Big(\sup_{s \ge 0} s^{k'-2} |V(s)|_{2k'+2}^2 + \int_0^\infty s^{k'-2} |V(s)|_{2k'+4}^2 \, \d s\Big).
\]
For $0 < \delta \le 1$ we define the ball $\cB_{k,\delta} \coloneqq \{V \colon \tbar V \tbar_k \le \delta\}$. For $V \in \cB_{k,\delta}$ observe that by estimate~\eqref{est-nonlinearity} of Proposition~\ref{prop-nonlinear} it holds $N[V] \in L^2(0,\infty;\cH^4)$. Define the mapping $\cT[V]$ as the unique solution to
\[
\partial_s (\cT[V]) + \cL (\cT[V]) = N[V] \quad \text{for} \quad s > 0, \quad \text{subject to} \quad \cT[V] = V^{(0)} \quad \text{at} \quad s = 0,
\]
given by Proposition~\ref{prop-maxreg}. We first consider the case $k = 2$. It then holds for $V \in \cB_{2,\delta}$, $C = C(\ell) < \infty$ with $C \delta < 1$, and $|V^{(0)}|_6 \le \delta^2$,
\begin{equation}\label{est-t}
\tbar \cT[V] \tbar_2 \stackrel{\eqref{a-priori-lin}}{\lesssim_\ell} |V^{(0)}|_6 + \lVert N[V] \rVert_{L^2(0,\infty;\cH^4)} \stackrel{\eqref{est-nonlinearity}}{\lesssim_\ell} \delta^2 + \delta^2 (1-C \delta)^{-1},
\end{equation}
where we have used Proposition~\ref{prop-maxreg} and Proposition~\ref{prop-nonlinear} above. Hence, $\cT[V] \in \cB_{2,\delta}$ provided $\delta \ll_\ell 1$. This proves that $\cT \colon \cB_{2,\delta} \to \cB_{2,\delta}$ is a self-map. In the same way we can deduce for $V, W \in \cB_{2,\delta}$,
\[
\tbar \cT[V] - \cT[W] \tbar_2 = \tbar \cT[V-W] \tbar_2 \stackrel{\eqref{a-priori-lin}}{\lesssim_\ell} \lVert N[V] - N[W] \rVert_{L^2(0,\infty;\cH^4)} \stackrel{\eqref{est-nonlinearity}}{\lesssim_\ell} \delta (1-C \delta)^{-1} \tbar V - W \tbar_2,
\]
proving that for $\delta \ll_\ell 1$ the map $\cT \colon \cB_{2,\delta} \to \cB_{2,\delta}$ is a contraction and thus has a unique fixed point $V \in \cB_{2,\delta}$, that is, a solution to \eqref{nonlinear-cauchy}. This proves existence in the minimal setting $k = 2$ by means of the contraction-mapping theorem.

\medskip

Now suppose that $V$ is the above constructed solution and $W$ with $\tbar W \tbar_2 < \infty$ is another solution to \eqref{nonlinear-cauchy}. By uniqueness of the linear Cauchy problem (cf.~Proposition~\ref{prop-maxreg}) we have for $\chi_\eps(s) \coloneqq 1$ if $s \le \eps$ and $\chi_\eps(s) = 0$ if $s > \eps$, where $0 < \eps < \infty$, due to Proposition~\ref{prop-nonlinear},
\begin{align*}
\tbar \chi_\eps (V - W) \tbar_2 &= \tbar \chi_\eps (\cT[V] - \cT[W]) \tbar_2 = \tbar \chi_\eps \cT[V - W] \tbar_2 \stackrel{\eqref{a-priori-lin}}{\lesssim_\ell} \lVert \chi_\eps (N[V] - N[W]) \rVert_{L^2(0,\infty;\cH^4)} \\
&\stackrel{\mathclap{\eqref{est-nonlinearity}}}{\lesssim_\ell} (\tbar \chi_\eps V \tbar_2 + \tbar \chi_\eps W \tbar_2) (1-C (\tbar \chi_\eps V \tbar_2 + \tbar \chi_\eps W \tbar_2))^{-1} \tbar \chi_\eps (V - W) \tbar_2,
\end{align*}
provided $C (\tbar \chi_\eps V \tbar_2 + \tbar \chi_\eps W \tbar_2) < 1$. Since $\lim_{\eps \searrow 0} \tbar \chi_\eps V \tbar_2 = \lim_{\eps \searrow 0} \tbar \chi_\eps W \tbar_2 = |V^{(0)}|_6 \le \delta^2$ it follows for $\delta \ll_\ell 1$ and $\eps \ll_\ell 1$ that $V(s,x) = W(s,x)$ for $0 \le s \le \eps$ and $-\ell \le x \le \ell$. On the other hand, estimate~\eqref{est-t} entails $|V(s)|_6 \le \delta^2$ on choosing $|V^{(0)}|_6$ sufficiently small, and hence we must have $V = W$ for $|V^{(0)}|_6 \ll_\ell 1$ by time-translation invariance.

\medskip

In order to prove higher regularity and the {\it \`a-priori} estimate \eqref{a-priori-nonlin}, note that as in \eqref{est-t} we obtain
\[
\tbar V \tbar_2 \lesssim_\ell |V^{(0)}|_6 + \delta (1-C \delta)^{-1} \tbar V \tbar_2 \quad \Rightarrow \quad \tbar V \tbar_2 \lesssim_\ell |V^{(0)}|_6 \le \delta^2
\]
for $\delta \ll_\ell 1$. Thus we obtain for $V \in \cB_{3,\delta}$ with help of Proposition~\ref{prop-maxreg} and Proposition~\ref{prop-nonlinear},
\begin{align*}
\tbar \cT[V] \tbar_3 &\stackrel{\eqref{a-priori-lin}}{\lesssim_\ell} |V^{(0)}|_6 + \lVert N[V] \rVert_{L^2(0,\infty;\cH^4)} + \lVert s \mapsto s^{\frac 1 2} N[V(s)] \rVert_{L^2(0,\infty;\cH^6)} \\
&\stackrel{\eqref{est-nonlinearity}}{\lesssim_\ell} |V^{(0)}|_6 + \frac{\|V\|_{C^0([0,\infty);\cH^6)}}{1-C \|V\|_{C^0([0,\infty);\cH^6)}} \big(\lVert V \rVert_{L^2(0,\infty;\cH^8)} + \lVert s \mapsto s^{\frac 1 2} V(s) \rVert_{L^2(0,\infty;\cH^{10})}\big) \\
&\lesssim_\ell |V^{(0)}|_6 + \delta (1-C\delta)^{-1} 
\tbar V \tbar_3
.
\end{align*}
Hence, for $\delta \ll_\ell 1$ there exists $C = C(\ell) < \infty$ such that
\[
\tbar \cT[V] \tbar_3 \le C |V^{(0)}|_6 + \frac 1 2 \tbar V \tbar_3.
\]
For an arbitrary element $V_0 \in \cB_{3,\delta}$ define $V_{j+1} \coloneqq \cT[V_j]$. Then we obtain
\[
\tbar V_j \tbar_3 \le C (2-2^{-j}) |V^{(0)}|_6 + 2^{-j} \tbar V_0 \tbar_3.
\]
Since as $j \to \infty$ the sequence $(V_j)_j$ converges to the unique solution $V$ to \eqref{nonlinear-cauchy} in $\tbar \cdot \tbar_2$, we infer by weak lower semi-continuity of the norm that $\tbar V \tbar_3 \le 2 C |V^{(0)}|_6$. We now boot-strap this argument and assume that $\tbar V(s_{k-1,\ell}+\cdot) \tbar_{k-1} \lesssim_{k,\ell} |V^{(0)}|_6$. We then obtain for $k \ge 4$, $s_{k,\ell} > s_{k',\ell}$ for all $k' < k$, and $V \in \cB_{k,\delta}$,
\begin{align*}
\tbar \cT[V(\cdot+s_{k,\ell})] \tbar_{k} &\stackrel{\eqref{a-priori-lin}}{\lesssim_\ell} |V^{(0)}|_6 + \sum_{k' = 2}^k \lVert s \mapsto s^{\frac{k'-2}{2}} N[V(s+s_{k,\ell})] \rVert_{L^2(0,\infty;\cH^{2k'})} \\
&\stackrel{\eqref{est-nonlinearity}}{\lesssim_\ell} |V^{(0)}|_6 + \sum_{k' = 2}^k \frac{\|V(s_{k,\ell}+\cdot)\|_{C^0([0,\infty);\cH^{2k})}}{1-C \|V(s_{k,\ell}+\cdot)\|_{C^0([0,\infty);\cH^{2k})}} \lVert s \mapsto s^{\frac{k'-2}{2}} V(s+s_{k,\ell}) \rVert_{L^2(0,\infty;\cH^{2k'+4})} \\
&\lesssim_{k,\ell} |V^{(0)}|_6 + \frac{\|V(s_{k,\ell}+\cdot)\|_{C^0([0,\infty);\cH^{2k})}}{1-C \|V(s_{k,\ell}+\cdot)\|_{C^0([0,\infty);\cH^{2k})}} 
\tbar V(\cdot+s_{k,\ell}) \tbar_k
.
\end{align*}
We have $\|V(s_{k,\ell}+\cdot)\|_{C^0([0,\infty);\cH^{2k})} \to 0$ as $s_{k,\ell} \to \infty$ since $\tbar V(s_{k-1,\ell}+\cdot) \tbar_{k-1} < \infty$, so that for $s_{k,\ell}$ sufficiently large this entails as before $\tbar V(s_{k,\ell}+\cdot) \tbar_k \lesssim_{k,\ell} |V^{(0)}|_6$ for the unique solution $V$ to \eqref{nonlinear-cauchy}. The additional time regularity in \eqref{a-priori-nonlin} follows from using the nonlinear PDE \eqref{pde-cauchy} in conjunction with estimate~\eqref{hardy-lx} of Lemma~\ref{lem-hardy-lx} and estimate~\eqref{est-nonlinearity} of Proposition~\ref{prop-nonlinear}.
\end{proof}

\section{Stability}\label{sec-stability}
Here, we provide a rigorous stability proof based on applying Gr\"onwall's lemma.

\begin{proof}[Proof of Theorem~\ref{th-stability}]
In what follows, we assume $s \gg_{k,\ell} 1$ so that by \eqref{regularity-weighted-2} of Theorem~\ref{th-well} sufficient regularity is guaranteed. From \eqref{pde-cauchy} and $\cL^k$ applied, tested against $\cL^k V$, we get
\begin{equation*}
|\cL^k V(s)|_U^2 \le |\cL^k V(s_{k,\ell})|_U^2 - \frac 2 5 \int_{s_{k,\ell}}^s |\cL^k V|_U^2 \, \d s' + 2 \int_{s_{k,\ell}}^s (\cL^k N[V], \cL^k V)_U \, \d s'.
\end{equation*}
For $k \ge 1$ we get with Lemma~\ref{lem-coerc}, Lemma~\ref{lem-bounds-uphi}, and Lemma~\ref{lem-elliptic} for $s_{k,\ell} \gg_{k,\ell} 1$,
\begin{align*}
|\cL^k V(s)|_U^2 \qquad &\stackrel{\mathclap{\eqref{coercivity-l-2-1}}}{\le} \qquad |\cL^k V(s_{k,\ell})|_U^2 - \frac 2 5 \int_{s_{k,\ell}}^s |\cL^k V|_U^2 \, \d s' \\
&\phantom{\le} \qquad + 2 \int_{s_{k,\ell}}^s \int_{-\ell}^\ell \big(\tfrac 1 5 U (\cL^{k-1} N[V]) (\cL^k V) + \Phi (\partial_x \cL^{k-1} N[V]) (\partial_x \cL^k V) \\
&\phantom{\phantom{\le} \qquad + 2 \int_{s_{k,\ell}}^s \int_{-\ell}^\ell \big(} + U^2 (\partial_x^2 \cL^{k-1} N[V]) (\partial_x^2 \cL^k V)\big) \, \d x\, \d s' \\
&\stackrel{\mathclap{\eqref{est-u-ell^2-x^2},\eqref{est-l}}}{\le} \qquad |\cL^k V(s_{k,\ell})|_U^2 - \frac 2 5 \int_{s_{k,\ell}}^s |\cL^k V|_U^2 \, \d s' + C \int_{s_{k,\ell}}^s |N[V]|_{4k-2} |V|_{4k+2} \, \d s' \\
&\stackrel{\mathclap{\eqref{est-nonlinearity}}}{\le} \qquad |\cL^k V(s_{k,\ell})|_U^2 - \frac 2 5 \int_{s_{k,\ell}}^s |\cL^k V|_U^2 \, \d s' + C \int_{s_{k,\ell}}^s \frac{|V|_{\max\{2k+2,6\}} |V|_{4k+2}^2}{1-C|V|_{\max\{2k+2,6\}}} \, \d s',
\end{align*}
where the nonlinear estimate, Proposition~\ref{prop-nonlinear}, has been used in the last step and we have used the fact that $|V|_{\max\{2k+2,6\}} \ll_{k,\ell} 1$ (which is true for $|V^{(0)}|_6 \ll_\ell 1$ and $s \gg_{k,\ell} 1$ by the nonlinear {\it \`a-priori} estimate~\eqref{a-priori-nonlin} of Theorem~\ref{th-well}). Hence, Gr\"onwall's lemma entails
\begin{equation}\label{decay-clmv}
|\cL^k V(s)|_U^2 \le \Big(|\cL^k V(s_{k,\ell})|_U^2 + C \int_{s_{k,\ell}}^s \frac{|V|_{\max\{2k+2,6\}} |V|_{4k+2}^2}{1-C|V|_{\max\{2k+2,6\}}} \, \d s'\Big) e^{- \frac 2 5 s}.
\end{equation}
With \eqref{a-priori-nonlin} of Theorem~\ref{th-well} we infer that $\int_{s_{k,\ell}}^\infty \frac{|V|_{\max\{2k+2,6\}}^2 |V|_{4k+2}^2}{(1-C |V|_{\max\{2k+2,6\}})^2} \, \d s < \infty$, so that with \eqref{est-l} of Lemma~\ref{lem-elliptic} (elliptic regularity) we have
\begin{equation*}
|V(s)|_{4k} \sim_{k,\ell} |\cL^k V(s)|_U \lesssim_{k,\ell} e^{- \frac 1 5 s} \quad \text{for} \quad s \gg_{k,\ell} 1,
\end{equation*}
which implies \eqref{est-decay} for $m = 0$. For $m = 1$ this follows from \eqref{est-decay} for $m = 0$ in conjunction with the PDE \eqref{pde-cauchy}, \eqref{hardy-lx} of Lemma~\ref{lem-hardy-lx}, and the nonlinear estimate \eqref{est-nonlinearity} of Proposition~\ref{prop-nonlinear}.
\end{proof}
\addcontentsline{toc}{section}{References}
\bibliographystyle{plain}
\bibliography{gnann-ibrahim}

\end{document}